\documentclass{amsart}
\usepackage{amsfonts,amssymb,amsmath,amsthm,mathrsfs}
\usepackage{url}
\usepackage{enumerate}

\usepackage[pdftex, pagebackref, bookmarksnumbered, bookmarksopen, colorlinks, citecolor=blue, linkcolor=blue]{hyperref}

\newtheorem{theorem}{Theorem}[section]
\newtheorem{lemma}[theorem]{Lemma}

\theoremstyle{definition}

\newtheorem{remark}[theorem]{Remark}
\numberwithin{equation}{section}

\author[G. Hu]{Guoen Hu}
\address{Guoen Hu:  Department of Applied Mathematics, Zhengzhou Information Science and Technology Institute, Zhengzhou 450001,
 People's Republic of China}
\email{guoenxx@163.com}
\author[X. Tao]{Xiangxing Tao}
\address{Xiangxing Tao: Department of Mathematics, School of Science, Zhejiang University of Science and Technology,
Hangzhou 310023, People's Republic of China}
\email{xxtao@zust.edu.cn}
\thanks{The research of the first author was supported by the NNSF of
China (Nos. 11871108, 11971295), and the research of the second (corresponding)  author was supported by
the NNSF of
China (No. 11771399).}
\begin{document}

\title[commutator]{An endpoint estimate for the  commutators of   singular integral operators with rough kernels}

\begin{abstract}
Let $\Omega$ be homogeneous of degree zero and have mean value zero on the unit sphere ${S}^{d-1}$, $T_{\Omega}$  be the homogeneous singular integral operator  with kernel $\frac{\Omega(x)}{|x|^d}$ and $T_{\Omega,\,b}$ be the commutator of $T_{\Omega}$ with symbol $b$. In this paper, we prove that if $\Omega\in L(\log L)^2(S^{d-1})$, then for $b\in {\rm BMO}(\mathbb{R}^d)$, $T_{\Omega,\,b}$ satisfies an endpoint estimate of $L\log L$ type.
\end{abstract}
\maketitle
\section{Introduction}
In this paper, we will work on $\mathbb{R}^d$, $d\geq 2$. Let $T$ be a linear operator from $\mathcal{S}(\mathbb{R}^d)$ to $\mathcal{S}'(\mathbb{R}^d)
$ and $b\in L^1_{{\rm loc}}(\mathbb{R}^d)$. The commutator of $T$ with symbol $b$, is defined by
$$T_bf(x)=b(x)Tf(x)-T(bf)(x).$$
A celebrated result of Coifman,  Rochberg and  Weiss \cite{crw} states  that if $T$ is a Calder\'on-Zygmund operator, then $T_{b}$ is bounded on $L^p(\mathbb{R}^d)$ for every $p\in (1,\,\infty)$ and also a converse result in
terms of the Riesz transforms. P\'erez \cite{perez1} considered the weak type endpoint estimate for the commutator of Calder\'on-Zygmund operator, and  proved the following result.
\begin{theorem}\label{thmperez}
Let $T$ be a Calder\'on-Zygmund operator and $b\in {\rm BMO}(\mathbb{R}^d)$. Then for any $\lambda>0$,
$$|\{x\in\mathbb{R}^d:\,|T_{b}f(x)|>\lambda\}|\lesssim\int_{\mathbb{R}^d}\Phi\big(\frac{|f(x)|}{\lambda}\big)dx,$$
where and in the following, $\Phi(t)=t\log({\rm e}+t)$.
\end{theorem}
Let $\Omega$ be
homogeneous of degree zero, integrable and have mean value zero on
the unit sphere ${S}^{d-1}$. Define the  singular integral operator
${T}_{\Omega}$ by
\begin{eqnarray}\label{eq1.1}{T}_{\Omega}f(x)={\rm p.\,v.}\int_{\mathbb{R}^d} \frac
{\Omega( y')}{|y|^d}f(x-y)dy,\end{eqnarray}  where and in the following, $y'=y/|y|$ for $y\in\mathbb{R}^d$. This operator was introduced by
Calder\'on and Zygmund \cite{cz1}, and has been proved to be bounded on $L^p(\mathbb{R}^d)$, $ 1< p<\infty$, under various
assumptions on the homogeneous function $\Omega$. For instance,
Calder\'on and Zygmund \cite {cz2} proved that
if $\Omega\in L\log L({S}^{d-1})$, then $T_{\Omega}$ is bounded on
$L^p(\mathbb{R}^d)$ for $p\in (1,\,\infty)$.   Ricci
and Weiss \cite{rw} improved the result of Calder\'on-Zygmund, and
showed that $\Omega\in H^1(S^{d-1})$ guarantees the
$L^p(\mathbb{R}^d)$ boundedness on $L^p(\mathbb{R}^d)$ for $p\in
(1,\,\infty)$. Seeger \cite{se} showed that $\Omega\in L\log
L(S^{d-1})$ is a sufficient condition such that $T_{\Omega}$ is bounded
from $L^1(\mathbb{R}^d)$ to $L^{1,\,\infty}(\mathbb{R}^d)$.
For other works about the mapping properties of $T_{\Omega}$, we refer to the papers
\cite{chr2,duo,drf,ler4,rw} and the references therein.

We now consider the commutator of $T_{\Omega}$ with symbol in ${\rm BMO}(\mathbb{R}^d)$.
Let $p\in [1,\,\infty)$ and $w$ be a nonnegative, locally integrable function on $\mathbb{R}^d$. We say that   $w\in A_{p}(\mathbb{R}^d)$ if
$$[w]_{A_p}=\sup_{Q}\Big(\frac{1}{|Q|}\int_Qw(x)dx\Big)\Big(\frac{1}{|Q|}\int_{Q}w^{1-p'}(x)dx\Big)^{p-1}<\infty,\,\,\,p\in (1,\,\infty),$$
the  supremum is taken over all cubes in $\mathbb{R}^d$,  $p'=p/(p-1)$, and $w\in A_1(\mathbb{R}^d)$ if
$${\rm ess}\sup_{x\in\mathbb{R}^d}\frac{Mw(x)}{w(x)}<\infty,$$ see \cite[Chapter 9]{gra} for the properties of $A_p(\mathbb{R}^d)$.
By the result of Duandikoetxea and Rubio de Francia \cite{drf} (see also \cite{duo}), we know that if $\Omega\in L^q(S^{d-1})$ for some $q\in (1,\,\infty]$, then for $p\in (q',\,\infty)$ and $w\in A_{p/q'}(\mathbb{R}^d)$
$$\|T_{\Omega}f\|_{L^p(\mathbb{R}^d,\,w)}\lesssim_{d,p,w}\|f\|_{L^p(\mathbb{R}^d,\,w)}.$$
This, together with \cite[Theorem 2.13]{abkp}, tells us that if $\Omega\in L^q(S^{d-1})$ for $q\in (1,\,\infty]$, then for $b\in {\rm BMO}(\mathbb{R}^d)$,
$$\|T_{\Omega,\,b}f\|_{L^p(\mathbb{R}^d,\,w)}\lesssim_{d,p,w}\|b\|_{{\rm BMO}(\mathbb{R}^d)}\|f\|_{L^p(\mathbb{R}^d,\,w)},\,\,p\in (q',\,\infty),\,\,w\in A_{p/q'}(\mathbb{R}^d).$$
Hu \cite{hu1} proved that if $\Omega\in L(\log L)^2(S^{d-1})$, then $T_{\Omega,\,b}$ is bounded on $L^p(\mathbb{R}^d)$ for all $p\in (1,\,\infty)$, see also \cite{hsw} for the $L^p(\mathbb{R}^d)$ boundedness of $T_{\Omega,\,b}$ when $\Omega$ satisfies another minimum size condition.

The weak type endpoint estimates of $T_{\Omega,\,b}$ are of interest. By Theorem \ref{thmperez}, we know that if $\Omega\in {\rm Lip}_{\alpha}(S^{d-1})$ with $\alpha\in (0,\,1]$ and $b\in {\rm BMO}(\mathbb{R}^d)$, then for any $\lambda>0$,
\begin{eqnarray}\label{eq1.2}|\{x\in\mathbb{R}^d:\,|T_{\Omega,\,b}f(x)|>\lambda\}|\lesssim\int_{\mathbb{R}^d}\Phi\big(\frac{|f(x)|}{\lambda}\big)dx.\end{eqnarray}
Recently, Lan, Tao and Hu \cite{lth} established the  weak type endpoint estimates for $T_{\Omega,\,b}$ when $\Omega$ satisfies only size condition. They proved that

\begin{theorem}\label{thm1.1}
Let $\Omega$ be homogeneous of degree zero and have mean value zero on $S^{d-1}$, $b\in {\rm BMO}(\mathbb{R}^d)$. Suppose that $\Omega\in L^{q}(S^{d-1})$  for some $q\in (1,\,\infty]$, then for any   $\lambda>0$ and weight $w$ such that $w^{q'}\in A_1(\mathbb{R}^d)$,
\begin{eqnarray*}&&w\big(\{x\in\mathbb{R}^d:\,|T_{\Omega,\,b}f(x)|>\lambda\}\big)\lesssim_{d,\,w}\int_{\mathbb{R}^d}\Phi\big(\frac{D|f(x)|}{\lambda}\big)w(x)dx,
\end{eqnarray*}
with $D=\|\Omega\|_{L^{q}(S^{d-1})}\|b\|_{{\rm BMO}(\mathbb{R}^d)}.$
\end{theorem}

The purpose of this paper is to give a weak type endpoint estimate of $T_{\Omega,\,b}$ when $\Omega$ satisfies certain minimum size condition. For a function $\Omega$ on $S^{d-1}$ and $\kappa\geq 0$, we say that $\Omega\in L(\log L)^{\kappa}(S^{d-1})$, if
$$\|\Omega\|_{L(\log L)^{\kappa}(S^{d-1})}:=\int_{S^{d-1}}|\Omega(x')\log^{\kappa}({\rm e}+|\Omega(x')|)|dx'<\infty.$$
Our main result can be stated as follows.
\begin{theorem}\label{thm1.2}Let $\Omega$ be homogeneous of degree zero, have mean value zero on $S^{d-1}$ and $\Omega\in L(\log L)^2({S}^{d-1})$, $T_{\Omega}$ be the operator defined by (\ref{eq1.1}). Then for $b\in {\rm BMO}(\mathbb{R}^d)$ and $\lambda>0$,
\begin{eqnarray}\label{eq1.weak}|\{x\in\mathbb{R}^d:\, |T_{\Omega,\,b}f(x)|>\lambda\}|\lesssim \int_{\mathbb{R}^d}\Phi\big(\frac{|f(x)|}{\lambda}\big)dx.\end{eqnarray}
\end{theorem}

\begin{remark}For $r\in (1,\,\infty)$, let $\mathscr{M}_{r,\,T_{\Omega}}$ be the maximal operator defined by
\begin{eqnarray}\label{eq1.4}\mathscr{M}_{r,\,T_{\Omega}}f(x)=\sup_{Q\ni x}\Big(\frac{1}{|Q|}\int_Q|T_{\Omega}(f\chi_{\mathbb{R}^d\backslash 3Q})(\xi)|^rd\xi\Big)^{1/r},\end{eqnarray}
where the supremum is taken over all cubes $Q\subset \mathbb{R}^d$ containing $x$. This operator was introduced by Lerner \cite{ler4}, who proved that for any $r\in (1,\,\infty)$,
\begin{eqnarray}\label{eq1.5}\|\mathscr{M}_{r,\,T_{\Omega}}f\|_{L^{1,\,\infty}(\mathbb{R}^d)}\lesssim r\|\Omega\|_{L^{\infty}(S^{d-1})}\|f\|_{L^1(\mathbb{R}^d)},
\end{eqnarray}
see \cite[Lemma 3.3]{ler4}. The crucial estimate in the proof of Theorem \ref{thm1.1} is
\begin{eqnarray}\label{eq1.6}\|\mathscr{M}_{r,\,T_{\Omega}}f\|_{L^{1,\,\infty}(\mathbb{R}^d)}\lesssim_r \|f\|_{L^1(\mathbb{R}^d)},
\end{eqnarray}when $\Omega\in L^q(S^{d-1})$ for some $q>1$.
However, the  estimate (\ref{eq1.6}) does not hold  and the argument used in \cite{lth} does not applies when $\Omega\in L(\log L)^2(S^{d-1})$. In fact, as in the proof of Theorem 1.2 in \cite{lth}, the estimate (\ref{eq1.6}) implies the $L^p(\mathbb{R}^d,\,w)$ boundedness of $T_{\Omega}$ for large $p\in (1,\,\infty)$  and  $w\in A_{s}(\mathbb{R}^n)$ for some $s\geq 1$, which is impossible when    $\Omega\in L(\log L)^2(S^{d-1})$ (see \cite[Theorem 1]{mucken}). To prove Theorem \ref{thm1.2}, we will employ some ideas and estimates of Ding and Lai \cite{dinglai} (see also Seeger \cite{se}).  However, the estimate for $T_{\Omega,\,b}h$ is much more complicated and more refined than  the estimate of $T_{\Omega}h$ in \cite{dinglai,se}, here $h$ is the bad part in the Calder\'on-Zygmund decomposition of function $f$.   Some   computations
of Luxmberg norms, interpolation between Orlicz spaces, an observation  of Hyt\"onen and P\'erez \cite{hp} and the interpolation with changes of measures, are involved in the estimate $T_{\Omega,\,b}h$; see Lemma \ref{lem2.1}, Lemma \ref{lem3.2} and Lemma \ref{lem4.3} for details.
\end{remark}

\begin{remark}
Let $\widetilde{T}_{\Omega}$ be the operator defined by
\begin{eqnarray}\label{eq1.7}\widetilde{T}_{\Omega}f(x)={\rm p.\,v.}\int_{\mathbb{R}^d}\Omega(x-y)K(x,\,y)f(y)dy.\end{eqnarray}
Suppose that $\widetilde{T}_{\Omega}$ is bounded on $L^2(\mathbb{R}^d)$. For $b\in {\rm BMO}(\mathbb{R}^d)$, define the commutator of $\widetilde{T}_{\Omega}$ by
\begin{eqnarray}\label{eq1.8}
\widetilde{T}_{\Omega,\,b}f(x)=b(x)\widetilde{T}_{\Omega}f(x)-\widetilde{T}_{\Omega}(bf)(x)
\end{eqnarray}
initially for $f\in\mathcal{S}(\mathbb{R}^d)$. Mimicking the proof of Theorem \ref{thm1.2}, we can   prove the following  result.
\begin{theorem}\label{thm1.3}Let $\Omega$ be homogeneous of degree zero, have mean value zero on $S^{d-1}$ and $\Omega\in L(\log L)^{2}({S}^{d-1})$, $\widetilde{T}_{\Omega}$ be the operator defined by (\ref{eq1.7}) and $\widetilde{T}_{\Omega,\,b}$ be the commutator defined by (\ref{eq1.8}). Suppose that $\widetilde{T}_{\Omega}$ and $\widetilde{T}_{\Omega,\,b}$ are bounded on $L^2(\mathbb{R}^d)$,  $K$ satisfies the size condition that
$$|K(x,\,y)|\lesssim \frac{1}{|x-y|^d},
$$
and the regularity that for some $\delta\in (0,\,1]$,
$$|K(x_1,\,y)-K(x_2,\,y)|\lesssim \frac{|x_1-x_2|^{\delta}}{|x_1-y|^{d+\delta}},\,\,|x_1-y|\geq 2|x_1-x_2|,
$$
$$|K(x,\,y_1)-K(x,\,y_2)|\lesssim \frac{|y_1-y_2|^{\delta}}{|x-y_1|^{d+\delta}},\,\,|x-y_1|\geq 2|y_1-y_2|.
$$
Then for $b\in {\rm BMO}(\mathbb{R}^d)$, $k\in\mathbb{N}$ and $\lambda>0$,
$$|\{x\in\mathbb{R}^d:\, |\widetilde{T}_{\Omega,\,b}f(x)|>\lambda\}|\lesssim \int_{\mathbb{R}^d}\Phi\big(\frac{|f(x)|}{\lambda}\big) dx.$$
\end{theorem}
As it was pointed out in \cite{dinglai}, Theorem \ref{thm1.3} is more general than Theorem \ref{thm1.2}.
\end{remark}

This paper is organized as follows. In Section 2, we outline some known facts about Orlicz spaces, and give a lemma concerning the interpolation between Orlicz spaces. In Section 3, we reduce the proof of Theorem \ref{thm1.2} to the proof of two key estimates   (\ref{equa3.u2}) and  (\ref{equa3.u3}). In Section 4 and Section 5, we prove (\ref{equa3.u2}) and  (\ref{equa3.u3}) respectively.

 Throughout this paper, $C$ always denotes a
positive constant that is independent of the main parameters
involved but whose value may differ from line to line. We use the
symbol $A\lesssim B$ to denote that there exists a positive constant
$C$ such that $A\le CB$.  Specially, we use $A\lesssim_{d,p} B$ to denote that there exists a positive constant
$C$ depending only on $d,p$ such that $A\le CB$. Constant with subscript such as $C_1$, does not change in different occurrences.
For any set $E\subset\mathbb{R}^d$,
$\chi_E$ denotes its characteristic function.  For a cube
$Q\subset\mathbb{R}^d$ and $\lambda\in(0,\,\infty)$, we use $\ell(Q)$   to denote the side length   of $Q$, and
$\lambda Q$ to denote the cube with the same center as $Q$ and whose
side length is $\lambda$ times that of $Q$. For a local function $b$ and a cube $Q$, $\langle b\rangle_Q$ denotes the mean value of $b$ on $Q$.

\section{Preliminary results on Orlicz spaces}
In this section, we list some known facts about Orlicz spaces. These facts can be found in \cite{rr}. Let $\Psi:\, [0,\,\infty)\rightarrow [0,\,\infty)$ be Young function, namely, $\Psi$ is convex and continuous on $[0,\,\infty)$, $\Psi(0)=0$ and $\lim_{t\rightarrow\infty}\Psi(t)=\infty$. We always assume  that $\Psi$ satisfies a doubling condition, that is, $\Psi(2t)\leq C\Psi(t)$ for any $t\in (0,\,\infty)$. A Young function $\Psi$ is called an $N$-function, if $\Psi(t)=0$ only in $t=0$, and $$\lim_{t\rightarrow 0}\frac{\Psi(t)}{t}=0,\,\,\,\lim_{t\rightarrow\infty}\frac{\Psi(t)}{t}=\infty.$$

Let $\Psi$ be a Young function, and $Q\subset \mathbb{R}^d$ be a cube. Define the space $L^{\Psi}(Q)$ as
$$L^{\Psi}(Q)=\{f:\,f\,\,\hbox{is\,\, measurable\,\,on}\,\,Q,\,\|f\|_{L^{\Psi}(Q)}<\infty\},
$$
with $\|\cdot\|_{L^{\Psi}(Q)}$ the Luxemburg norm defined by $$\|f\|_{L^{\Psi}(Q)}=\inf\Big\{\lambda>0:\,\frac{1}{|Q|}\int_{Q}\Psi\Big(\frac{|f(x)|}{\lambda}\Big)dx\leq 1\Big\}.$$
Then we have
$$\frac{1}{|Q|}\int_{Q}\Psi(|f(x)|)dx\leq 1\Leftrightarrow\|f\|_{L^{\Psi}(Q)}\leq 1,
$$
see \cite[p. 54]{rr}. Also, we have that and
$$\|f\|_{L^{\Psi}(Q)}\leq \inf\Big\{\lambda+\frac{\lambda}{|Q|}\int_{Q}\Psi\Big(\frac{|f(x)|}{\lambda}\Big)dx:\,\lambda>0\Big\}\leq 2\|f\|_{L^{\Psi}(Q)};$$
see \cite[p. 69]{rr}.

Let $\Psi$ be a Young function. We define its complementary function $\Psi^*$ on $[0,\,\infty)$ by
$$\Psi^*(t)=\sup\{st-\Psi(s): \,s\geq 0\}.
$$
Then $\Psi^*$ is also a Young function. We have that \begin{eqnarray}\label{eq2.holder}t_1t_2\leq \Psi(t_1)+\Psi^*(t_2),\,\,t_1,\,t_2\in [0,\,\infty),\end{eqnarray}
and consequently, the generalized H\"older inequality
$$\frac{1}{|Q|}\int_Q|f(x)h(x)|dx\leq \|f\|_{L^{\Psi}(Q)}\|h\|_{L^{\Psi^*}(Q)}
$$holds for $f\in L^{\Psi}(Q)$ and $h\in L^{\Psi^*}(Q)$.
see \cite[p. 6]{rr}. Also, we have
\begin{eqnarray}\label{eq2.twonorms}C\|f\|_{L^{\Psi}(Q)}\leq \sup_{\|h\|_{L^{\Psi^*}(Q)}\leq 1}\frac{1}{|Q|}\Big|\int_{Q}f(x)h(x)dx\Big|\leq \|f\|_{L^{\Psi}(Q)},
\end{eqnarray}
see  inequality (18) in \cite[p.\,62]{rr}.
When the  functions $\Psi$ and $\Psi^*$ are $N$-functions,  the inequality
$$t\leq \Psi^{-1}(t) (\Psi^*)^{-1}(t)\leq 2t,$$
holds true for all $t>0$, where $\Psi^{-1}(t)$ is the inverse  of $\Psi(t)$ (see \cite[p.13]{rr} for details).

Now let $p\in [1,\,\infty)$ and $\alpha\in \mathbb{R}$, set $\Phi_{p,\,\alpha}(t)=t^p\log^{\alpha}({\rm e}+t)$. Note that $\Phi_{p,\,p}(t)=(\Phi(t))^p.$ As it is well known, for $p\in (1,\,\infty)$ and $\alpha\in [0,\,\infty)$,
the complementary function of $\Phi_{p,\,\alpha}$ is
$$\Phi_{p,\,\alpha}^*(t)\approx t^{p'}\log^{-\alpha/(p-1)}({\rm e}+t),$$
see \cite{oneil}. Usually, we denote $\|f\|_{L^{\Phi_{p,\,\alpha}}(Q)}$ as $\|f\|_{L^p(\log L)^{\alpha},\,Q}$. Observe that when $p\in (1,\,\infty)$, $\Phi_{p,\,\alpha}(t)$  satisfies the doubling condition.

As it is well known, for $\Phi(t)=t\log ({\rm e}+t)$, we have that $\Phi^*(t)\approx {\rm e}^t-1$.
For a cube $Q\subset \mathbb{R}^d$, we also define $\|f\|_{{\rm exp}L,\,Q}$ by
$$\|f\|_{{\rm exp}L,\,Q}=\inf\Big\{t>0:\,\frac{1}{|Q|}\int_{Q}\Phi^*\Big(\frac{|f(y)|}{t}\Big)dy\leq 1\Big\}.$$
Let $b\in{\rm BMO}(\mathbb{R}^d)$. The John-Nirenberg inequality tells us that for any $Q\subset \mathbb{R}^d$,
$$\|b-\langle b\rangle_Q\|_{{\rm exp}L,\,Q}\lesssim \|b\|_{{\rm BMO}(\mathbb{R}^d)}.$$
This, together with  the generalization of H\"older's inequality, shows that
\begin{eqnarray}\label{eq1.final}
\frac{1}{|Q|}\int_{Q}|b(x)-\langle b\rangle_Q||h(x)|dx\lesssim \|h\|_{L\log L,\,Q}\|b\|_{{\rm BMO}(\mathbb{R}^d)}.
\end{eqnarray}

The following lemma will be used in the proof of Theorem \ref{thm1.2}.

\begin{lemma}\label{lem2.1}
Let $Q\subset \mathbb{R}^d$ be a cube,  $p\in (1,\,\infty)$, $\alpha\in [0,\,\infty)$ and $C_1\in (0,\,1]$. Suppose that
$$\frac{1}{|Q|}\int_{Q}|f(y)|dy\leq C_1,\,\,\|f\|_{L^p(\log L)^{-\alpha},\,Q}\leq 1.$$
Then for $q\in (1,\,p)$, $r\in (0,\,1)$ such that $1/q=r+(1-r)/p$, and $\varepsilon\in (0,\,r)$,
$$\Big(\frac{1}{|Q|}\int_{Q}|f(y)|^qdy\Big)^{\frac{1}{q}}\lesssim C_1^{\varepsilon}.
$$
\end{lemma}
\begin{proof} At first, we claim that for  $q_1\in [1,\,p)$,
\begin{eqnarray}\label{equation2.4}
\Big(\frac{1}{|Q|}\int_{Q}|h(y)|^{q_1}dy\Big)^{\frac{1}{q_1}}\lesssim \|h\|_{L^p(\log L)^{-\alpha},\,Q}.
\end{eqnarray}
To prove this, we assume that $\|h\|_{L^p(\log L)^{-\alpha},\,Q}=1$, which means that
$$\frac{1}{|Q|}\int_{Q}\Phi_{p,\,-\alpha}(|h(x)|)dx\leq 1.
$$
Observe that when $t\in [1,\,\infty)$, $t^{q_1-p}\lesssim \log^{-\alpha}({\rm e}+t)$. Therefore,
\begin{eqnarray*}
\frac{1}{|Q|}\int_{Q}|h(y)|^{q_1}dy&\leq& 1+\frac{1}{|Q|}\int_{\{y\in Q:\,|h(y)|\geq 1\}}|h(y)|^{q_1}dy\\
&\lesssim &1+\frac{1}{|Q|}
\int_{\{y\in Q:\,|h(y)|\geq 1\}}\Phi_{p,-\alpha}(|h(y)|)dy\lesssim 1.
\end{eqnarray*}
This verifies (\ref{equation2.4}).
For fixed  $q\in (1,\,\infty)$ and $\varepsilon\in (0,\,r)$, we choose $q_1\in (1,\,p)$ such that  $1/q=\varepsilon+(1-\varepsilon)/q_1$. It then follows from (\ref{equation2.4}) that
$$\Big(\int_{Q}|f(y)|^qdy\Big)^{\frac{1}{q}}\le  \Big(\int_{Q}|f(y)|dy\Big)^{\varepsilon}\Big(\int_{Q}|f(y)|^{q_1}dy\Big)^{\frac{1-\varepsilon}{q_1}}
\lesssim C_1^{\varepsilon}|Q|^{1/q},$$
and then completes the proof of Lemma \ref{lem2.1}.
\end{proof}

\section{Proof of Theorem \ref{thm1.2}}
In this section, we will  start to prove Theorem \ref{thm1.2}. In particular, we reduces its proof to two estimates (\ref{equa3.u2}) and  (\ref{equa3.u3}), which will be proved in Section 4 and Section 5 respectively.

To prove Theorem \ref{thm1.2}, we will employ the well known micro-local decomposition introduced by Seeger \cite{se}, see \cite[Section 2]{dinglai} for its variant. For $s>3$, let  $\mathfrak{E}^s=\{e_{\nu}^s\}_{\nu\in\Lambda_s}$ be a collection of unit vectors on $S^{d-1}$ such that
\begin{itemize}
\item[\rm (a)] $|e_{\nu}^s-e_{\nu'}^s|>2^{-s\gamma-4}$ when $\nu\not=\nu'$;
\item[\rm (b)] for each $\theta\in S^{d-1}$, there exists an $e^s_{\nu}$ such that $|e^s_{\nu}-\theta|\leq 2^{-s\gamma-4},$
\end{itemize}
where $\gamma\in (0,\,1)$ is a constant.  The set $\mathfrak{E}^s$ can be constructed as in \cite[Section 2]{dinglai}.  Observe that ${\rm card}(\mathfrak{E}^s)\lesssim 2^{s\gamma(d-1)}$.
Let $\zeta$  be a smooth, nonnegative, radial function, such that ${\rm supp}\, \zeta\subset B(0,\,1)$ and $\zeta(t)=1$ for $|t|\leq 1/2$. Set
$$\widetilde{\Gamma}_\nu^s(\xi)=\zeta\Big(2^{s\gamma}\big(\frac{\xi}{|\xi|}-e^s_{\nu}\big)\Big)$$
and$$\Gamma_{\nu}^s(\xi)=\widetilde{\Gamma}^s_{\nu}(\xi)\Big(\sum_{\nu\in\Lambda_s}\widetilde{\Gamma}^s_{\nu}(\xi)\Big)^{-1}.$$ It is easy to verify that $\Gamma^s_{\nu}$ is homogeneous of degree zero, and for all $s$,
$$\sum_{\nu\in\Lambda_s}\Gamma^s_{\nu}(\xi)=1,\,\, \xi\in S^{d-1}.
 $$

Let $\psi\in C^{\infty}_0(\mathbb{R})$ such that $0\leq \psi\leq 1$,  ${\rm supp}\, \psi\subset [-4,\,4]$ and $\psi(t)\equiv 1$ when $t\in [-2,\,2]$. Define the multiplier operator $G_{\nu}^s$ by
$$\widehat{G_{\nu}^sf}(\xi)=\psi\big(2^{s\gamma}\langle \xi/|\xi|, e_{\nu}^s\rangle\big)\widehat{f}(\xi),
$$
where and in the following, for a suitable function $f$, $\hat{f}$ denotes the Fourier transform of $f$.
Take a smooth radial nonnegative function $\phi$ on $\mathbb{R}^d$ such that ${\rm supp}\, \phi\subset\{x:\frac{1}{2}\leq |x|\leq 2\}$ and
$\sum_j\phi_j(x)=1$ for all $x\in \mathbb{R}^d\backslash\{0\}$, where $\phi_j(x)=\phi(2^{-j}x)$.

Recall that, $\mathcal{D}$,  the standard dyadic grid in $\mathbb{R}^d$ consists of all cubes of the form $$2^{-k}([0,\,1)^d+l),\,k\in  \mathbb{Z},\,\,l\in\mathbb{Z}^d.$$
For $j\in\mathbb{Z}$, let $\mathcal{D}_j=\{Q\in \mathcal{D}:\, \ell(Q)=2^j\}$.

\medskip

{\it Proof of Theorem \ref{thm1.2}}. By homogeneity, it suffices to prove  (\ref{eq1.weak}) for the case of $\lambda=1$.
Applying the Calder\'on-Zygmund decomposition to $\Phi(|f|)$ at level $1$, we can obtain a collection
of non-overlapping closed dyadic cubes $\mathcal{S}=\{Q\}$, such that $\|f\|_{L^{\infty}(\mathbb{R}^d\backslash \cup_{Q\in\mathcal{S}}Q)}\lesssim 1$,  and
$$\int_{Q}\Phi(|f(x)|)dx \lesssim |Q|,\,\,\,\sum_{Q\in\mathcal{S}}|Q|\lesssim \int_{\mathbb{R}^d}\Phi(|f(x)|)dx.$$
Let  $E=\cup_{Q\in\mathcal{S}}2^{200}Q$, it is obvious that $|E|\lesssim \int_{\mathbb{R}^d}\Phi(|f(x)|)dx$. Set$$g(x)=f(x)\chi_{\mathbb{R}^d\backslash \cup_{Q\in\mathcal{S}}Q}(x)+\sum_{Q\in\mathcal{S}}\langle f\rangle_{Q}\chi_{Q}(x),
$$
and
$$h(x)=\sum_{Q\in\mathcal{S}}h_{Q}(x),\,\,\hbox{with}\,\,h_Q(x)=(f(x)-\langle f\rangle_{Q}\big)\chi_{Q}(x).$$
It is easy to verify that  for each cube $Q\in\mathcal{S}$,
$$\|h_{Q}\|_{L\log L,\,Q}\lesssim 1.
$$
By $L^2(\mathbb{R}^d)$ boundedness of $T_{\Omega,\,b}$,   we have that
\begin{eqnarray}\label{equation3.13}\big|\{x\in\mathbb{R}^d:\, |T_{\Omega,\,b}g(x)|>1/2\}\big|\lesssim \int_{\mathbb{R}^d}|f(x)| dx.\end{eqnarray}

Let
$$E_0=\{x'\in S^{d-1}:\, |\Omega(x')|\leq 1\}$$ and $$E_i=\{x'\in S^{d-1}:\,2^{i-1}<|\Omega(x')|\leq 2^i\}\,\,(i\in\mathbb{N}).$$
Denote
$$\Omega_0(x')=\Omega(x')\chi_{E_0}(x'),\,\,\Omega_i(x')=\Omega(x')\chi_{E_i}(x')\,\,(i\in\mathbb{N}).$$ Set $K_j(x)=\frac{\Omega(x')}{|x|^d}\phi_j(x)$, $K_j^i(x)=\frac{\Omega_i(x')}{|x|^d}\phi_j(x)$, $K_{j\nu}^{i,\,s}(x)=\frac{\Omega_i(x')}{|x|^d}\phi_j(x)\Gamma_{\nu}^s(x')$,  $T_j$ be the convolution operators with kernel $K_j$, and
$$T_j^iu(x)=K_j^i*u(x),\,\,\,T_{j\nu}^{i,s}u(x)=K_{j\nu}^{i,s}*u(x).$$
Observe that for each fixed $s$,
$T_{j}^{i}u(x)=\sum_{\nu}T_{j\nu}^{i,\,s}u(x)$.
It is obvious that  ${\rm supp}T_jh_Q\subset 2^{100}Q$ when $Q\in\mathcal{S}_{j-s}$ with $j\in\mathbb{Z}$ and $s<100$. Set $\mathcal{S}_j=\mathcal{D}_j\cap \mathcal{S}$.
For $x\in \mathbb{R}^d\backslash E$, we can decompose $T_{\Omega,\,b}h$ as
\begin{eqnarray*}T_{\Omega,\,b}h(x)&=&\sum_{Q}(b-\langle b\rangle_Q)T_{\Omega}h_Q(x)-T_{\Omega}\Big(\sum_{Q\in\mathcal{S}}(b-\langle b\rangle_Q)h_Q\Big)(x)\\
&=&\sum_{s\geq 100}\sum_{j\in\mathbb{Z}}\sum_{Q\in \mathcal{S}_{j-s}}(b(x)-\langle b\rangle_Q)T_jh_Q(x)-T_{\Omega}\Big(\sum_{Q\in\mathcal{S}}(b-\langle b\rangle_Q)h_Q\Big)(x).\end{eqnarray*}
Recall that $T_{\Omega}$ is bounded from $L^1(\mathbb{R}^d)$ to $L^{1,\,\infty}(\mathbb{R}^d)$. An application of (\ref{eq1.final}) tells us that
\begin{eqnarray}\label{equation3.14}
&&\Big|\Big\{x\in\mathbb{R}^d:\Big|T_{\Omega}\Big(\sum_{Q\in\mathcal{S}}(b-\langle b\rangle_Q)h_Q\Big)(x)\Big|>\frac{1}{4}\Big\}\Big|\\
&&\quad\lesssim \sum_{Q\in\mathcal{S}}\|(b-\langle b\rangle_Q)h_Q\|_{L^1(\mathbb{R}^d)}\lesssim\sum_{Q\in\mathcal{S}}|Q|
\|h_{Q}\|_{L\log L,Q}\nonumber\\
&&\quad\lesssim \int_{\mathbb{R}^d}\Phi(|f(x)|)dx.\nonumber
\end{eqnarray}
With estimates (\ref{equation3.13}) and (\ref{equation3.14}) in hand, it suffices to prove that
\begin{eqnarray}\label{equation3.15}
&&\Big|\Big\{x\in\mathbb{R}^d\backslash E:\,\Big|\sum_{s\geq 100}\sum_{j\in\mathbb{Z}}\sum_{Q\in \mathcal{S}_{j-s}}(b(x)-\langle b\rangle_Q)T_jh_Q(x)\Big|>\frac{1}{4}\Big\}\Big|\lesssim \|f\|_{L^1(\mathbb{R}^d)}.
\end{eqnarray}

To prove (\ref{equation3.15}), let
$${\rm U}_1h(x)=\sum_{i=0}^{\infty}\sum_{100\leq s\leq N_0i}\sum_{j\in\mathbb{Z}}\sum_{Q\in\mathcal{S}_{j-s}}\big(b(x)-\langle b\rangle_Q\big)T^i_{j}h_Q(x),
$$
$${\rm U}_2h(x)=\sum_{i=0}^{\infty}\sum_{s> N_0i}\sum_{j\in\mathbb{Z}}\sum_{Q\in\mathcal{S}_{j-s}}\sum_{\nu}G_{\nu}^s\big[\big(b-\langle b\rangle_Q\big)T^{i,s}_{j\nu}h_Q\big](x),
$$
and
\begin{eqnarray*}
{\rm U}_3h(x)&=&\sum_{i=0}^{\infty}\sum_{s> N_0i}\sum_{j\in\mathbb{Z}}\sum_{Q\in\mathcal{S}_{j-s}}\Big[\big(b(x)-\langle b\rangle_Q\big)T^i_{j}h_Q(x)\\
&&\qquad-\sum_{\nu}G_{\nu}^s\big[\big(b-\langle b\rangle_Q\big)T^{i,s}_{j\nu}h_Q\big](x)\Big],
\end{eqnarray*}
where and in the following, $N_0\in\mathbb{N}$ is a constant which will be chosen in the estimate for ${\rm U}_2$ and ${\rm U}_3$, see (\ref{equa3.const}) in Section 5.
For $x\in \mathbb{R}^d\backslash E$,
we write
\begin{eqnarray*}\sum_{s\geq 100}\sum_{j\in\mathbb{Z}}\sum_{Q\in \mathcal{S}_{j-s}}(b(x)-\langle b\rangle_Q)T_jh_Q(x)={\rm U}_1h(x)+{\rm U}_2h(x)+{\rm U}_3h(x).
\end{eqnarray*}

To estimate term ${\rm U}_1$, we claim that for each cube $Q\in \mathcal{S}_{j-s}$,
\begin{eqnarray}\label{equa3.claim}&&\|(b-\langle b\rangle_Q)T_{j}^ih_Q\|_{L^1(\mathbb{R}^d)}\lesssim \big(2^{-i}+(i+s)\|\Omega_i\|_{L^1(S^{d-1})}\big)\|h_Q\|_{L^1(\mathbb{R}^d)},
\end{eqnarray}
To see this, let   $x_Q$ be the center of $Q$. It is easy to  see that ${\rm supp}\,T_jh_Q\subset B_{Q}:=B(x_Q,\, 10d2^{j})$, and $|\langle b\rangle_Q-\langle b\rangle_{B_Q}|\lesssim s.$ Observing that for each $y\in Q$ and $\lambda>0$,
$$\int_{B_Q}
\frac{|\Omega_i(x-y)|}{\lambda}\log \Big({\rm e}+\frac{|\Omega_i(x-y)|}{\lambda}\Big)dx\lesssim 2^{jd}\int_{S^{d-1}}\frac{|\Omega_i(\theta)|}{\lambda}\log \Big({\rm e}+\frac{|\Omega_i(\theta)|}{\lambda}\Big)d\theta,
$$
we thus get that for $y\in Q$,
\begin{eqnarray*}\|\Omega_i(\cdot-y)\|_{L\log L,\,B_Q}&\lesssim & \inf\Big\{\lambda>0:\,\frac{\|\Omega_i\|_{L^1(S^{d-1})}}{\lambda}\log\Big({\rm e}+\frac{\|\Omega_i\|_{L^{\infty}(S^{d-1})}}{\lambda}\Big)\leq 1\Big\}\\
&\lesssim&\|\Omega_i\|_{L^{\infty}(S^{d-1})}^{-1}+\|\Omega_i\|_{L^1(S^{d-1})}\log ({\rm e}+\|\Omega_i\|_{L^{\infty}(S^{d-1})})\\
&\lesssim&2^{-i}+i\|\Omega_i\|_{L^1(S^{d-1})}.
\end{eqnarray*}
It then follows from   inequality (\ref{eq1.final}) that for each $y\in Q$,
\begin{eqnarray*}
\int_{B_Q}|K_j^i(x-y)||b(x)-\langle b\rangle_Q|dx&\le&2^{-jd}\int_{B_Q}|\Omega_i(x-y)||b(x)-\langle b\rangle_{B_Q}|dx\\
&&+2^{-jd}\int_{B_Q}|\Omega_i(x-y)|dx|\langle b\rangle_Q-\langle b\rangle_{B_Q}|\\
&\lesssim&2^{-i}+(i+s)\|\Omega_i\|_{L^1(S^{d-1})}.
\end{eqnarray*}
This, via duality argument, verifies (\ref{equa3.claim}). Now we obtain from (\ref{equa3.claim}) that
\begin{eqnarray*}
\|{\rm U}_1h\|_{L^1(\mathbb{R}^d)}&\le &\sum_{i=0}^{\infty}\sum_{100\leq s\leq N_0i}\sum_{j\in\mathbb{Z}}\sum_{Q\in \mathcal{S}_{j-s}}\|(b-\langle b\rangle_Q)T_{j}^ih_Q\|_{L^1(\mathbb{R}^d)}\\
&\lesssim&\sum_{i=0}^{\infty}\sum_{100\leq s\leq N_0i}\big(2^{-i}+(i+s)\|\Omega_i\|_{L^1(S^{d-1})}\big)\int_{\mathbb{R}^d}|f(x)|dx\\
&\lesssim&\big(1+\|\Omega\|_{L(\log L)^2(S^{d-1})}\big)\int_{\mathbb{R}^d}|f(x)|dx.
\end{eqnarray*}  
Therefore,
$$|\{x\in\mathbb{R}^d\backslash E:\, |{\rm U}_1h(x)|>\frac{1}{12}\}|\lesssim \int_{\mathbb{R}^d}|f(x)|dx.$$

The proof of (\ref{equation3.15}) is now reduced to proving that
\begin{eqnarray}\label{equa3.u2}\big|\big\{x\in\mathbb{R}^d\backslash E:\, |{\rm U}_2h(x)|>\frac{1}{12}\big\}\big|\lesssim \int_{\mathbb{R}^d}|f(x)|dx,
\end{eqnarray}
and
\begin{eqnarray}\label{equa3.u3}\big|\big\{x\in\mathbb{R}^d\backslash E:\, |{\rm U}_3h(x)|>\frac{1}{12}\big\}\big|\lesssim \int_{\mathbb{R}^d}|f(x)|dx.
\end{eqnarray}
The proofs of these two inequalities are long and complicated, and will be given in Section 4 and Section 5 respectively.\qed
\section{proof of inequality (\ref{equa3.u2})}

Let $\Omega$ be homogeneous of degree zero and $\Omega\in L^{\infty}(S^{d-1})$. For each $j\in\mathbb{Z}$ and $\nu\in \Lambda_s$, define operator $T_{j\nu}^s$  by
\begin{eqnarray}\label{equation3.1}
T_{j\nu}^sf(x)=K_{j\nu}^s*f(x),\end{eqnarray}where  $K_{j\nu}^s(x)=\Omega(x')|x|^{-d}\phi_j(x)\Gamma_{\nu}^s(x')$.
Let $\mathcal{S}$ be a collection of dyadic  cubes with disjoint interiors. For $m\in\mathbb{Z}$, let $\mathcal{S}_m=\mathcal{S}\cap \mathcal{D}_{m}$. Then for each $\nu$ and $s\geq 3$,
\begin{eqnarray}\label{equation3.2}&&\Big\|\sum_j\sum_{Q\in\mathfrak{Q}_{j-s}}T_{j\nu}^sh_Q\Big\|^2_{L^2(\mathbb{R}^d)}\lesssim 2^{-2\gamma s(d-1)}\|\Omega\|^2_{L^{\infty}(S^{d-1})}\sum_{j}\sum_{Q\in \mathfrak{Q}_{j-s}}\|h_Q\|_{L^1(\mathbb{R}^d)},
\end{eqnarray}
where $\mathfrak{Q}_{j-s}\subset \mathcal{S}_{j-s}$,  each $h_Q$ is supported on cube $Q\in\mathfrak{Q}_{j-s}$ and  $\|h_Q\|_{L^1(\mathbb{R}^d)}\le |Q|.$ This fact was proved in \cite[p.1658]{dinglai} (also \cite[p.\,99]{se}) and plays an important role in the weak type endpoint estimate for $T_{\Omega}$.

To prove  inequality (\ref{equa3.u2}), we need the following key lemma
which can be considered as a  refined version of the estimate (\ref{equation3.2}).
\begin{lemma}\label{lem3.2}
Let $\Omega$ be homogeneous of degree zero and $\Omega\in L^{\infty}(S^{d-1})$, $\mathcal{S}$ be a collection of dyadic   cubes with disjoint interiors. For
each cube $Q\in\mathcal{S}$, let $h_Q$ be an integrable function supported in $Q$ satisfying $\|h_Q\|_{L^1(\mathbb{R}^d)}\leq |Q|$.
Then  for $b\in{\rm BMO}(\mathbb{R}^d)$ and $s\geq 100$,
$$\Big\|\sum_{j}\sum_{Q\in\mathcal{S}_{j-s}}\sum_{\nu}G_{\nu}^s\big((b-\langle b\rangle_Q)T_{j\nu}^sh_Q\big)\Big\|^2_{L^2(\mathbb{R}^d)}\lesssim \|\Omega\|_{L^{\infty}(S^{d-1})}^22^{-s\gamma/2}\sum_{Q\in\mathcal{S}}\|h_{Q}\|_{L^1(\mathbb{R}^d)}.$$
\end{lemma}
\begin{proof} For   $f\in L^2(\mathbb{R}^d)$, it follows from Cauchy-Schwarz inequality that
\begin{eqnarray*}
&&\Big|\sum_{j}\sum_{Q\in\mathcal{S}_{j-s}}\sum_{\nu}\int_{\mathbb{R}^d}G_{\nu}^s\big((b-\langle b\rangle_Q)T_{j\nu}^sh_Q\big)(x)f(x)dx\Big|\\
&&\quad=\Big|\int_{\mathbb{R}^d}
\sum_{\nu}G_{\nu}^sf(x)\sum_{j}\sum_{Q\in\mathcal{S}_{j-s}}(b(x)-\langle b\rangle_Q)T_{j\nu}^sh_Q(x)dx\Big|\\
&&\quad\leq\Big\|\Big(\sum_{\nu}|G_{\nu}^sf|^2\Big)^{\frac{1}{2}}\Big\|_{L^2(\mathbb{R}^d)}
\Big(\sum_{\nu}\Big\|\sum_{j}\sum_{Q\in\mathcal{S}_{j-s}}(b-\langle b\rangle_Q)T_{j\nu}^sh_Q\Big\|^2_{L^2(\mathbb{R}^d)}\Big)^{\frac{1}{2}}.
\end{eqnarray*}
Plancherel's theorem, via the estimate
$$\sup_{\xi\not=0}\sum_{\nu}|\psi(2^{s\gamma}\langle e^s_{\nu},\xi/|\xi|\rangle)|^2\lesssim 2^{s\gamma(d-2)}
$$
(see \cite[inequality (3.1)]{dinglai}, implies that
\begin{eqnarray}\label{equation3.3}\Big\|\Big(\sum_{\nu}|G_{\nu}^sf|^2\Big)^{\frac{1}{2}}\Big\|_{L^2(\mathbb{R}^d)}^2 &=&\sum_{\nu}\int_{\mathbb{R}^d}|\psi(2^{s\gamma}\langle \xi/|\xi|, e_{\nu}^s\rangle)|^2|\widehat{f}(\xi)|^2d\xi\\ &\lesssim&2^{s\gamma(d-2)}\|f\|_{L^2(\mathbb{R}^d)}^2.\nonumber\end{eqnarray}
Recall that ${\rm card}(\mathfrak{E}^s)\lesssim 2^{\gamma s(d-1)}$. It suffices to prove that for each fixed $\nu\in\Lambda_s$,
\begin{eqnarray}\label{equation3.4}
&&\Big\|\sum_{j}\sum_{Q\in\mathcal{S}_{j-s}}(b-\langle b\rangle_Q)T_{j\nu}^sh_Q\Big\|^2_{L^2(\mathbb{R}^d)}\lesssim  2^{-s\gamma(2d-\frac{5}{2})}\|\Omega\|^2_{L^{\infty}(S^{d-1})}\sum_{Q\in\mathcal{S}}\|h_Q\|_{L^1(\mathbb{R}^d)}.
\end{eqnarray}
By homogeneity, we may assume that $\|\Omega\|_{L^{\infty}(S^{d-1})}=\|b\|_{{\rm BMO}(\mathbb{R}^d)}=1.$

We now prove (\ref{equation3.4}).   Write
\begin{eqnarray}\label{equation3.4'}
&&\Big\|\sum_{j}\sum_{Q\in\mathcal{S}_{j-s}}(b-\langle b\rangle_Q)T_{j\nu}^sh_Q\Big\|^2_{L^2(\mathbb{R}^d)}\\
&&\quad=\sum_{j}\sum_{Q\in\mathcal{S}_{j-s}}\sum_{I\in\mathcal{S}_{j-s}}\int_{\mathbb{R}^d}h_Q(x)
T_{j\nu}^s\Big(\big(b-\langle b\rangle_Q)(b-\langle b\rangle_{I}) T^s_{j\nu}h_{I}\Big)(x)dx\nonumber\\
&&\quad+2\sum_{j}\sum_{Q\in\mathcal{S}_{j-s}}\sum_{i<j}\sum_{I\in\mathcal{S}_{i-s}}\int_{\mathbb{R}^d}h_Q(x)
T_{j\nu}^s\Big(\big(b-\langle b\rangle_Q)(b-\langle b\rangle_{I})T^s_{i\nu}h_{I}\Big)(x)dx.\nonumber
\end{eqnarray}
For each fixed $j,\,\nu$ and $s$, let $$\widetilde{\mathcal{R}}_{j\nu}^s=\{y\in\mathbb{R}^d:\, |\langle y,\,e^s_{\nu}\rangle|\leq 2^{j+2},\,|y-\langle y,\,e_{\nu}^s\rangle e^{s}_{\nu}|\leq 2^{j+2-s\gamma}\},$$
and
$${\mathcal{R}}_{j\nu}^s=\widetilde{\mathcal{R}}_{j\nu}^s+\widetilde{\mathcal{R}}_{j\nu}^s.$$   As it was pointed out by Seeger \cite[p.\,99]{se} (see also Ding and Lai \cite[p.\,1659]{dinglai}), when $i\leq j$, we have that \begin{eqnarray*}&&\sum_{I\in\mathcal{S}_{i-s}}T_{j\nu}^s\Big(\big(b-\langle b\rangle_Q)(b-\langle b\rangle_{I})T^s_{i\nu}h_{I}\Big)(x)\\
&&\quad=\sum_{I\in \mathcal{S}_{i-s}, \atop{I\cap \{x+\mathcal{R}_{j\nu}^s\}\not=\emptyset}} \int_{\mathbb{R}^d}K_{j\nu}^s(x-y)(b(y)-\langle b\rangle_Q)(b(y)-\langle b\rangle_I)T_{i\nu}^sh_I(y)dy.
\end{eqnarray*}Observe that $$|x+2\mathcal{R}_{j\nu}^s|\lesssim 2^{jd-\gamma s(d-1)}.$$
For each fixed $Q\in\mathcal{S}_{j-s}$ and $x\in Q$, we can find a cube ${R}_{j,s}^x$ centered at $x$, such that $Q\subset {R}_{j,s}^x$,
$|{R}_{j,s}^x|\approx 2^{jd}$, and
$$\bigcup_{i\leq j}\bigcup_{I\in \mathcal{S}_{i-s}\atop{I\cap \{x+\mathcal{R}_{j\nu}^s\}\not=\emptyset}}I \subset x+2\mathcal{R}_{j\nu}^s\subset  R_{j,\,s}^{x}.$$
For each fixed $i\leq j$, $I\in \mathcal{S}_{i-s}$, let $I^s=2^{s+4}dI$. Then $|\langle b\rangle_I-\langle b\rangle_{I^s}|\lesssim s.$
Observe that for each $r\in [1,\,\infty)$,
$$\|b-\langle b\rangle_{I^s}\|_{L^{r'}(I^s)}\lesssim 2^{id/r'},\,\,\|K_{i\nu}^s\|_{L^r(\mathbb{R}^d)}\lesssim 2^{-\gamma s(d-1)/r}2^{-id/r'},$$
and that
$$\sum_{i\leq j}\sum_{I\in \mathcal{S}_{i-s}, \atop{I\cap \{x+\mathcal{R}_{j\nu}^s\}\not=\emptyset}}\|h_I\|_{L^1(\mathbb{R}^d)}\lesssim
\sum_{i\leq j}\sum_{I\in \mathcal{S}_{i-s}, \atop{I\cap \{x+\mathcal{R}_{j\nu}^s\}\not=\emptyset}}|I|\lesssim 2^{jd-\gamma s(d-1)}.
$$
Recall that ${\rm supp}\, K^s_{i\nu}\subset \{x:\,|x|\leq 2^{i+2}\}$. Thus for each $I\in \mathcal{S}_{i-s}$, ${\rm supp}\,T^s_{i\nu}h_I\subset I^s$. A trivial computation involving H\"older's inequality  gives us that
\begin{eqnarray}\label{equation3.5}
&&\Big\|\sum_{i\leq j}\sum_{I\in \mathcal{S}_{i-s}, \atop{I\cap \{x+\mathcal{R}_{j\nu}^s\}\not=\emptyset}}|b-\langle b\rangle_I||T_{i\nu}^sh_I|\Big\|_{L^1(\mathbb{R}^d)}\\
&&\quad\lesssim \sum_{i\leq j}\sum_{I\in \mathcal{S}_{i-s}, \atop{I\cap \{x+\mathcal{R}_{j\nu}^s\}\not=\emptyset}}\big(s\|T_{i\nu}^sh_I\|_{L^1(\mathbb{R}^d)}+\|(b-\langle b\rangle_{I^s})T_{i\nu}^sh_I\|_{L^1(\mathbb{R}^d)}\big)\nonumber\\
&&\quad\lesssim s2^{-\gamma s(d-1)}\sum_{i\leq j}\sum_{I\in \mathcal{S}_{i-s}, \atop{I\cap \{x+\mathcal{R}_{j\nu}^s\}\not=\emptyset}}\|h_I\|_{L^1(\mathbb{R}^d)}\nonumber\\
&&\qquad+\sum_{i\leq j}\sum_{I\in \mathcal{S}_{i-s}, \atop{I\cap \{x+\mathcal{R}_{j\nu}^s\}\not=\emptyset}}\|b-\langle b\rangle_{I^s}\|_{L^{r'}(I^s)}
\|K_{i\nu}^s\|_{L^{r}(\mathbb{R}^d)}\|h_I\|_{L^1(\mathbb{R}^d)} \nonumber\\
&&\quad\lesssim 2^{-2\gamma s(d-1)/r}2^{jd}.\nonumber
\end{eqnarray}
Now we claim that for $p\in (1,\,\infty)$,
\begin{eqnarray}\label{equation3.6}
\Big\|\sum_{i\leq j}\sum_{I\in \mathcal{S}_{i-s}, \atop{I\cap \{x+\mathcal{R}_{j\nu}^s\}\not=\emptyset}}\big|(b-\langle b\rangle_I)T_{i\nu}^sh_I\big|\Big\|_{L^p(\log L)^{-p},\,R_{j,s}^x}\lesssim s.
\end{eqnarray}
To prove this, let ${\rm supp}\,f\subset R_{j,s}^x$ with $\|f\|_{L^{p'}(\log L)^{p'},\,R_{j,s}^x}=1$, namely,
$$\int_{R_{j,s}^x}\Phi_{p',\,p'}(|f(z)|)dz\leq |R_{j,s}^x|.$$
Let $M$ be the Hardy-Littlewood maximal operator.  A straightforward computation involving inequality (\ref{eq2.holder}) leads to that for $I\in\mathcal{S}_{i-s}$ and $y\in I$,
\begin{eqnarray*}
|K_{i\nu}^s|*\big(|b-\langle b\rangle_{I}||f|\big)(y)&\lesssim &s|K_{i\nu}^s|*|f|(y)+|K_{i\nu}^s|*\big(|b-\langle b\rangle_{I^s}||f|\big)(y)\\
&\lesssim&|K_{i\nu}^s|* {\rm exp}\Big(\frac{|b-\langle b\rangle_{I^s}|}{C\|b\|_{{\rm BMO}(\mathbb{R}^d)}}\Big) (y)+s|K_{i\nu}^s|*(\Phi(|f|))(y)\\
&\lesssim &1+s\inf_{z\in I}M(\Phi(f))(z).
\end{eqnarray*}
Recall that ${\rm supp}\,f\subset R_{j,s}^x$, we then have that
$$\int_{R_{j,s}^x}M(\Phi(f))(y)dy\lesssim 2^{jd/p}\|M(\Phi(|f|))\|_{L^{p'}(\mathbb{R}^d)}\lesssim 2^{jd/p}\|\Phi(|f|)\|_{L^{p'}(\mathbb{R}^d)}\lesssim 2^{jd}.
$$
Therefore,
\begin{eqnarray*}
&&\sum_{i\leq j}\sum_{I\in \mathcal{S}_{i-s}, \atop{I\cap \{x+\mathcal{R}_{j\nu}^s\}\not=\emptyset}}
\int_{R_{j,s}^x}\big|f(y)(b(y)-\langle b\rangle_I)T_{i\nu}^sh_I(y)\big|dy\\
&&\quad\leq\sum_{i\leq j}\sum_{I\in \mathcal{S}_{i-s}, \atop{I\cap \{x+\mathcal{R}_{j\nu}^s\}\not=\emptyset}}
\big\||h_I||K_{i\nu}^s|*\big(|b-\langle b\rangle_I||f|\big)\big\|_{L^1(\mathbb{R}^d)}\\
&&\quad\lesssim s\sum_{i\leq j}\sum_{I\in \mathcal{S}_{i-s}, \atop{I\cap \{x+\mathcal{R}_{j\nu}^s\}\not=\emptyset}}
\big\|h_I\|_{L^1(\mathbb{R}^d)}+\sum_{i\leq j}\sum_{I\in \mathcal{S}_{i-s}, \atop{I\cap \{x+\mathcal{R}_{j\nu}^s\}\not=\emptyset}}|I|\inf_{z\in I}M(\Phi(f))(z))\\
&&\quad\lesssim s2^{-\gamma s(d-1)}2^{jd}+s\int_{R_{j,s}^x}M(\Phi(f))(y)dy\lesssim s2^{jd}.
\end{eqnarray*}
This, via inequality (\ref{eq2.twonorms}) leads to  (\ref{equation3.6}).

Inequalities (\ref{equation3.5}) and (\ref{equation3.6}), via Lemma \ref{lem2.1}, state that for each fixed $\varepsilon\in (0,\,1)$, we
can choose $q\in (1,\,2)$ which is close to $1$ sufficiently, such that
$$\Big\|\sum_{i\leq j}\sum_{I\in \mathcal{S}_{i-s}, \atop{I\cap \{x+\mathcal{R}_{j\nu}^s\}\not=\emptyset}}\big|(b-\langle b\rangle_I)T_{i\nu}^sh_I\big|\Big\|_{L^q(R_{j,s}^x)}\lesssim 2^{jd/q}2^{-2\varepsilon\gamma s(d-1)}.
$$
Let $j\in\mathbb{Z}$, $Q\in\mathcal{S}_{j-s}$ and $x\in Q$. Another application of H\"older's inequality yields
\begin{eqnarray*}&&\sum_{i\leq j}\sum_{I\in \mathcal{S}_{i-s}, \atop{I\cap \{x+\mathcal{R}_{j\nu}^s\}\not=\emptyset}} \Big|\int_{\mathbb{R}^d}K_{j\nu}^s(x-y)(b(y)-\langle b\rangle_Q)(b(y)-\langle b\rangle_I)T_{i\nu}^sh_I(y)dy\Big|\\
&&\quad \lesssim 2^{-jd}\Big(\int_{R_{j,s}^x}|b(y)-\langle b\rangle_Q|^{q'}dy\Big)^{\frac{1}{q'}}\Big\|\sum_{i\leq j}\sum_{I\in \mathcal{S}_{i-s}, \atop{I\cap \{x+\mathcal{R}_{j\nu}^s\}\not=\emptyset}}\big|(b-\langle b\rangle_I)T_{i\nu}^sh_I\big|\Big\|_{L^q(R_{j,s}^x)}\\
&&\quad\lesssim s2^{-2\varepsilon\gamma s(d-1)}.
\end{eqnarray*}
since $|\langle b\rangle_Q-\langle b\rangle_{R_{j,s}^x}|\lesssim s$.
This, in turn, implies that
$$\Big\|\sum_{j}\sum_{Q\in\mathcal{S}_{j-s}}(b-\langle b\rangle_Q)T_{j\nu}^sh_Q\Big\|^2_{L^2(\mathbb{R}^d)}\lesssim s2^{-2\varepsilon\gamma s(d-1)}\sum_{Q}\|h_Q\|_{L^1(\mathbb{R}^d)}.$$
We choose $\varepsilon\in (0,\,1)$   such that $2\varepsilon (d-1)=2d-7/3$. The last estimate, along with (\ref{equation3.4'}), establishes (\ref{equation3.4}) and then completes the proof of Lemma \ref{lem3.2}.
\end{proof}

{\it Proof of the inequality (\ref{equa3.u2})}. It  follows from Lemma \ref{lem3.2} that
\begin{eqnarray*}&&|\{x\in\mathbb{R}^d\backslash E:\, |{\rm U}_2h(x)|>\frac{1}{12}\}|\le  \|{\rm U}_2h\|_{L^2(\mathbb{R}^d)}^2\\
&&\quad\leq \Big(\sum_{i=0}^{\infty}\sum_{s> N_0i}\Big\|\sum_j\sum_{Q\in\mathcal{S}_{j-s}}\sum_{\nu}G_{\nu}^s\big[\big(b-\langle b\rangle_Q\big)T^i_{j}h_Q\big]\Big\|_{L^2(\mathbb{R}^d)}\Big)^2\\
&&\quad\lesssim\Big(\sum_{i\geq 0}2^i\sum_{s>N_0i}2^{-s\gamma/4}\big(\sum_{Q}\|h_Q\|_{L^1(\mathbb{R}^d)}\big)^{\frac{1}{2}}\Big)^2\lesssim \int_{\mathbb{R}^d}|f(x)| dx,\end{eqnarray*}
if we choose $N_0\in \mathbb{N}$ and $\gamma\in (0,\,1)$  such that $N_0\gamma>16$. This proves (\ref{equa3.u2}).\qed

\section{proof of inequality (\ref{equa3.u3})}
To prove (\ref{equa3.u3}), we will employ some lemmas.
\begin{lemma}\label{lem4.1}
Let $m$ be a complex-valued bounded function on $\mathbb{R}^d\backslash \{0\}$ such that
$$|\partial^{\alpha}_{\xi}m(\xi)|\le A|\xi|^{-|\alpha|}
$$
for all multi indices $\alpha$ with $|\alpha|\leq \lfloor d/2\rfloor+1$, where and in the following,  $\lfloor d/2\rfloor$ denote the integer part of $d/2$. Let $T_m$ be the multiplier operator defined by
$$\widehat{T_mf}(\xi)=m(\xi)\widehat{f}(\xi).
$$
Then for $w\in A_2(\mathbb{R}^d)$, $T_m$ is bounded on $L^2(\mathbb{R}^d,w)$ with bound $C_{d,[w]_{A_2}}(\|m\|_{L^{\infty}(\mathbb{R}^d)}+A)$, and is bounded from $L^1(\mathbb{R}^d)$ to $L^{1,\,\infty}(\mathbb{R}^d)$ with bound $C_{d}(\|m\|_{L^{\infty}(\mathbb{R}^d)}+A)$.
\end{lemma}

The boundedness of $T_m$ on $L^2(\mathbb{R}^d,\,w)$ with $w\in A_2(\mathbb{R}^d)$ and from $L^1(\mathbb{R}^d,\,w)$ to $L^{1,\,\infty}(\mathbb{R}^d,\,w)$ with $w\in A_1(\mathbb{R}^d)$ was proved by Kurtz and Wheeden \cite{kurtzwh}. Repeating the proof of
Theorem 1 in \cite{kurtzwh}, we can verify the bound of $T_m$ on $L^2(\mathbb{R}^d,\,w)$ ($w\in A_2(\mathbb{R}^d)$) is less than $C_{d,[w]_{A_2}}(\|m\|_{L^{\infty}(\mathbb{R}^d)}+A)$, while the bound from  $L^1(\mathbb{R}^d)$ to $L^{1,\,\infty}(\mathbb{R}^d)$  is less than $C_{d}(\|m\|_{L^{\infty}(\mathbb{R}^d)}+A)$.

Let $\eta\in C^{\infty}_0(\mathbb{R}^d)$ be a radial function such that  ${\rm supp}\,\eta\subset\{|\xi|\leq 2\}$, $0\leq \eta\leq 1$ and $\eta(\xi)=1$ when $|\xi|\leq 1$.  Define $\varphi_k(\xi)=\eta(2^k\xi)-\eta(2^{k+1}\xi)$, then ${\rm supp}\,\varphi_k\subset \{2^{-k-1}\leq |\xi|\leq 2^{-k+1}\}$. Define  multuplier operators $V_k$ and $W_k$  by
$$\widehat{V_kf}(\xi)=\eta(2^k\xi)\widehat{f}(\xi),\,\,\widehat{W_kf}(\xi)=\varphi_k(\xi)\widehat{f}(\xi),
$$
respectively. Observe that for any $m\in\mathbb{Z}$,
$I=V_m+\sum_{k<m}W_k$.
\begin{lemma}\label{lem4.3} Let $b\in {\rm BMO}(\mathbb{R}^d)$. Under  the same hypothesis and notations as in Lemma \ref{lem3.2}, we have that  for $m\in\mathbb{Z}$ and   $s\geq 100$,
\begin{eqnarray}\label{equation3.7}\Big\|\sum_{j}\sum_{\nu}G_{\nu,\,b}^sT_{j\nu}^sH_{j-s}\Big\|^2_{L^2(\mathbb{R}^d)}\lesssim \|\Omega\|_{L^{\infty}(S^{d-1})}^22^{-s\gamma/2}\sum_{Q}\|h_Q\|_{L^1(\mathbb{R}^d)},\end{eqnarray}
\begin{eqnarray}\label{equation3.8}&&\Big\|\sum_{j}\sum_{\nu}G_{\nu,b}^sV_mT^s_{j\nu}H_{j-s}\Big\|^2_{L^2(\mathbb{R}^d)}\lesssim \|\Omega\|_{L^{\infty}(S^{d-1})}^22^{-s\gamma/2}\sum_{Q}\|h_Q\|_{L^1(\mathbb{R}^d)}.\end{eqnarray}
where and in the following, $G_{\nu,b}^s$ is the commutator of $G_{\nu}^s$ with $b$, and  for $j\in\mathbb{Z}$, $H_j(x)=\sum_{Q\in\mathcal{S}_j}h_Q(x).$
\end{lemma}
\begin{proof} For each fixed $f\in L^2(\mathbb{R}^d)$, we have by Cauchy-Schwarz inequality that
\begin{eqnarray*}
&&\Big|\sum_{\nu}\sum_j\int_{\mathbb{R}^d}G_{\nu,\,b}^sT_{j\nu}^sH_{j-s}(x)f(x)dx\Big|\\
&&\quad=\Big|
\sum_{\nu}\sum_j\int_{\mathbb{R}^d}G_{\nu,\,b}^sf(x)T_{j\nu}^sH_{j-s}(x)dx\Big|\\
&&\quad\leq\Big\|\Big(\sum_{\nu}|G_{\nu,\,b}^sf|^2\Big)^{1/2}\Big\|_{L^2(\mathbb{R}^d)}\Big\|\Big(\sum_{\nu}
\Big|\sum_jT_{j\nu}^sH_{j-s}\Big|^2\Big)^{1/2}\Big\|_{L^2(\mathbb{R}^d)}.
\end{eqnarray*}
It follows from (\ref{equation3.2}) that
$$\Big\|\Big(\sum_{\nu}\Big|\sum_jT_{j\nu}^sH_{j-s}\Big|^2\Big)^{1/2}\Big\|_{L^2(\mathbb{R}^d)}^2\lesssim 2^{-s\gamma(d-1)}\|\Omega\|^2_{L^{\infty}(S^{d-1})}\sum_{Q}\|h_Q\|_{L^1(\mathbb{R}^d)}.
$$
On the other hand, we have by Cauchy-Schwarz inequality  that
\begin{eqnarray*}
&&\Big|\sum_{\nu}\sum_j\int_{\mathbb{R}^d}G_{\nu,\,b}^sV_mT_{j\nu}^sH_{j-s}(x)f(x)dx\Big|\\
&&\quad\leq\Big\|\Big(\sum_{\nu}|G_{\nu,\,b}^sf|^2\Big)^{1/2}\Big\|_{L^2(\mathbb{R}^d)}
\Big\|\Big(\sum_{\nu}\Big|V_m\sum_jT_{j\nu}^sH_{j-s}\Big|^2\Big)^{1/2}\Big\|_{L^2(\mathbb{R}^d)}\\
&&\quad\leq \Big\|\Big(\sum_{\nu}|G_{\nu,\,b}^sf|^2\Big)^{1/2}\Big\|_{L^2(\mathbb{R}^d)}\Big(\sum_{\nu}
\Big\|\sum_jT_{j\nu}^sH_{j-s}\Big\|_{L^2(\mathbb{R}^d)}^2\Big)^{1/2},
\end{eqnarray*}
where in the last inequality, we have invoked Plancherel's theorem and the fact that $$\|V_mf\|_{L^2(\mathbb{R}^d)}=\|\widehat{V_mf}\|_{L^2(\mathbb{R}^d)}=\|\eta(2^m\cdot)\widehat{f}\|_{L^2(\mathbb{R}^d)}\leq \|f\|_{L^2(\mathbb{R}^d)}.$$ If we can prove that
\begin{eqnarray}\label{equation3.9}
\Big\|\Big(\sum_{\nu}|G_{\nu,\,b}^sf|^2\Big)^{1/2}\Big\|^2_{L^2(\mathbb{R}^d)}\lesssim 2^{s\gamma(d-1)-s\gamma/2}\|f\|_{L^2(\mathbb{R}^d)}^2,
\end{eqnarray}
the inequalities (\ref{equation3.7}) and (\ref{equation3.8}) then follow  from duality directly.

To prove (\ref{equation3.9}), we will employ an observation of Coifman, Rochberg and Weiss
(see \cite[pp. 620-621]{crw}), which shows that certain weighted $L^p(\mathbb{R}^d)$ estimates
for linear operators imply the $L^p(\mathbb{R}^d)$ estimates for the corresponding commutators, see also \cite[Section 7]{hp}. We present the details here mainly to make the bound clearer.
We can verify that
\begin{eqnarray}\label{eq3.multiplier}|\partial_{\xi}^{\alpha}\psi(2^{s\gamma}\langle e^s_{\nu},\,\xi/|\xi|\rangle)|\lesssim 2^{s\gamma(\lfloor\frac{d}{2}\rfloor+1)}|\xi|^{-|\alpha|},\,\,|\alpha|\leq \lfloor\frac{d}{2}\rfloor+1\end{eqnarray}
for all $\nu\in\Lambda_s$.
Let $w\in A_2(\mathbb{R}^d)$ such that $w^{1+\epsilon}\in A_{2}(\mathbb{R}^d)$
for $\epsilon=2d+6$. We then have by Lemma \ref{lem4.1}   that
\begin{eqnarray}\label{equation3.11}
\sum_{\nu\in\Lambda_s}\|G_{\nu}^sf\|_{L^2(\mathbb{R}^d,\,w^{1+\epsilon})}^2&
\lesssim_{d,[w^{1+\epsilon}]_{A_2}} &2^{s\gamma(d-1)}2^{2s\gamma(\lfloor\frac{d}{2}\rfloor+1)}\|f\|_{L^2(\mathbb{R}^d,\,w^{1+\epsilon})}^2\\
&\lesssim_{d,[w^{1+\epsilon}]_{A_2}}&2^{s\gamma(2d+1)}\|f\|_{L^2(\mathbb{R}^d, w^{1+\epsilon})}^2.\nonumber
\end{eqnarray}Note that $f\rightarrow \big(\sum_{\nu}|G_{\nu}^sf|^2\big)^{1/2}$ is sublinear.  Applying interpolation theorem of Stein and Weiss \cite{stwe}, we deduce from (\ref{equation3.3}) and (\ref{equation3.11}) that
\begin{eqnarray}\label{equation3.12}
\Big\|\Big(\sum_{\nu}|G_{\nu}^sf|^2\Big)^{\frac{1}{2}}\Big\|_{L^2(\mathbb{R}^d,w)}
&\lesssim_{d,[w^{1+\epsilon}]_{A_2}} &2^{s\gamma[\frac{d-2}{2}(1-t)+\frac{2d+1}{2}t]}\|f\|_{L^2(\mathbb{R}^d,w)}\\
&\lesssim_{d,[w^{1+\epsilon}]_{A_2}}&2^{s\gamma(\frac{d}{2}-\frac{3}{4})} \|f\|_{L^2(\mathbb{R}^d,w)}.\nonumber
\end{eqnarray}
with $t=\frac{1}{1+\epsilon}$. Now let $b\in {\rm BMO}(\mathbb{R}^d)$. \cite[Lemma 7.3]{hp} tells us that there exists a constant $c_d$ such that
$$[{\rm e}^{(1+\epsilon)2{\rm Re}zb}]_{A_2}\le_d C,\,\,\hbox{if}\,\,|z|\leq \frac{c_d}{2(1+\epsilon)\|b\|_{{\rm BMO}(\mathbb{R}^d)}}.$$
For $z\in\mathbb{C}$, let $$G^z_{\nu,\,s,\,b}f={\rm e}^{zb}G_{\nu}^s({\rm e}^{-zb}f).$$
It is obvious that
$$\|G^z_{\nu,\,s,\,b}f\|_{L^2(\mathbb{R}^d)}=\|G_{\nu}^s({\rm e}^{-zb}f)\|_{L^2(\mathbb{R}^d, {\rm e}^{2{\rm Re}zb})}
$$
As  in \cite[inequality (7.7)]{hp}, we choose $\rho=\frac{c_d}{4(1+\epsilon)\|b\|_{{\rm BMO}(\mathbb{R}^d)}} $ and  have that
\begin{eqnarray*}
\|G_{\nu,b}^sf\|_{L^2(\mathbb{R}^d)}&\le &\frac{1}{2\pi \rho^2}\int_{|z|=\rho}\|G^z_{\nu,\,s,\,b}f\|_{L^2(\mathbb{R}^d)}|dz|\\
&\leq& \frac{1}{2\pi \rho^{\frac{3}{2}}}\Big(\int_{|z|=\rho}\|G^z_{\nu,\,s,\,b}f\|^2_{L^2(\mathbb{R}^d)}|dz|\Big)^{\frac{1}{2}}.\end{eqnarray*}
It now follows from (\ref{equation3.12}) (with $w={\rm e}^{2{\rm Re}zb}$) that,
\begin{eqnarray*}
\sum_{\nu}\|G_{\nu,b}^sf\|^2_{L^2(\mathbb{R}^d)}&\le &\frac{1}{4\pi^2 \rho^3}\int_{|z|=\rho}\sum_{\nu}\|G_{\nu}^s({\rm e}^{-zb}f)\|^2_{L^2(\mathbb{R}^d, {\rm e}^{2{\rm Re}zb})}|dz|\\
&\lesssim&2^{s\gamma(d-\frac{3}{2})}\|f\|_{L^2(\mathbb{R}^d)}^2.
\end{eqnarray*}
This leads to (\ref{equation3.9})   and then completes the proof of Lemma \ref{lem4.3}.
\end{proof}

\begin{lemma}\label{lem4.4} For fixed $k,\,s,\,j,\,\nu$, let $K_{k,j\nu}^s(x,\,y)$ be the kernel of the operator $(I-G_{\nu}^s)W_kT_{j\nu}^s$, namely,
$$(I-G_{\nu}^s)W_kT_{j\nu}^sh(x)=\int_{\mathbb{R}^d}K_{k,j\nu}^s(x,\,y)h(y)dy.
$$
Under the hypothesis of Theorem \ref{thm1.2},
we have that for each $x,y\in \mathbb{R}^d$ and $N_1\in\mathbb{N}$,
\begin{eqnarray*}|K_{k,j\nu}^s(x,\,y)|&\lesssim_{N_1}& 2^{s\gamma(N_1+2N)}2^{(-j+k)N_1}2^{-j}2^{-kd}\|\Omega\|_{L^{\infty}(S^{d-1})}\\
&&\times \int_{S^{d-1}}|\Gamma_{\nu}^s(\theta)|\int_{2^{j-1}}^{2^{j+1}}\frac{1}{(1+2^{-2k}|x-y-r\theta|^2)^N}dr d\theta,\end{eqnarray*}
where and in the following, $N=\lfloor d/2\rfloor+1$.
\end{lemma}
\begin{proof} We follow the argument in \cite[Section 4.2]{dinglai}. Let $$L_{k,\,s,\,\nu}(\xi)=(1-\psi(2^{s\gamma}\langle \xi/|\xi|, e_{\nu}^s\rangle))\varphi_{k}(\xi),$$ and write
\begin{eqnarray*}
K_{k,j\nu}^s(x,\,y)&=&\frac{1}{(2\pi)^d}\int_{\mathbb{R}^d}\int_{\mathbb{R}^d}{\rm e}^{i\langle x-z,\xi\rangle}L_{k,s,\nu}(\xi)d\xi K^s_{j\nu}(z-y)dz\\
&=&\frac{1}{(2\pi)^d}\int_{S^{d-1}}\Omega(\theta)\Gamma_{\nu}^s(\theta)\\
&&\quad\Big\{
\int_{\mathbb{R}^d}\int^{\infty}_0{\rm e}^{i\langle x-y-r\theta,\xi\rangle}\phi(2^{-j}r)r^{-1}drL_{k,s,\nu}(\xi)d\xi\Big\}d\theta.
\end{eqnarray*}
Recall that ${\rm supp}\phi\subset \{1/2\leq |z|\leq 2\}$. Integrating by parts with $r$, we deduce that
\begin{eqnarray*}\int^{\infty}_0{\rm e}^{i\langle x-y-r\theta,\xi\rangle}\phi(2^{-j}r)r^{-1}dr
=\int^{\infty}_0{\rm e}^{i\langle x-y-r\theta,\xi\rangle}(i\langle \theta,\xi\rangle)^{-N_1}
(\phi(2^{-j}r)r^{-1})^{(N_1)}dr.
\end{eqnarray*}
On the other hand, integrating by parts with $\xi$ leads to that
\begin{eqnarray*}
&&\int_{\mathbb{R}^d}{\rm e}^{i\langle x-y-r\theta,\xi\rangle}(i\langle \theta,\xi\rangle)^{-N_1}L_{k,s,\,\nu}(\xi)d\xi\\
&&\quad=\int_{\mathbb{R}^d}{\rm e}^{i\langle x-y-r\theta,\,\xi\rangle} \frac{(I-2^{-2k}\Delta_{\xi})^N}{(1+2^{-2k}|x-y-r\theta|^2)^N}\big(L_{k,\,s,\,\nu}(\xi)(i\langle \theta,\,\xi\rangle)^{-N_1}\big)d\xi.
\end{eqnarray*}
Therefore,
\begin{eqnarray*}K_{k,j\nu}^s(x,\,y)&=&\frac{1}{(2\pi)^d}\int_{S^{d-1}}\Omega(\theta)\Gamma_{\nu}^s(\theta)
\int_{\mathbb{R}^d}{\rm e}^{i\langle x-y-r\theta,\,\xi\rangle}\int_{0}^{\infty}(\phi(2^{-j}r)r^{-1})^{(N_1)}\\
&&\times\frac{(I-2^{-2k}\Delta_{\xi})^N}{(1+2^{-2k}|x-y-r\theta|^2)^N}\big(L_{k,\,s,\,\nu}(\xi)(i\langle \theta,\,\xi\rangle)^{-N_1}\big)drd\xi d\theta.
\end{eqnarray*}
Since
$$|(I-2^{-2k}\Delta_{\xi})^N\big((\langle\theta,\xi\rangle)^{-N_1} L_{k,\,s,\,\nu}(\xi)\big)|\lesssim_{N_1}2^{(s\gamma+k)N_1+2s\gamma N},
$$
see  \cite[(4.12)]{dinglai}, we obtain  that
\begin{eqnarray*}
|K_{k,j\nu}^s(x,\,y)|&\lesssim_{N_1}&2^{(s\gamma+k)N_1+2s\gamma N}\int_{S^{d-1}}|\Omega(\theta)\Gamma_{\nu}^s(\theta)|\\
&&\times\int_{2^{-k-1}\leq |\xi|\leq 2^{-k+1}}
\int_{0}^{\infty}\frac{(\phi(2^{-j}r)r^{-1})^{(N_1)}dr}{(1+2^{-2k}|x-y-r\theta|^2)^N}d\xi d\theta\\
&\lesssim_{N_1} &2^{(s\gamma+k)N_1+2s\gamma N}2^{-j(N_1+1)}2^{-kd}\|\Omega\|_{L^{\infty}(S^{d-1})}\\
&&\times \int_{S^{d-1}}|\Gamma_{\nu}^s(\theta)|\int_{2^{j-1}}^{2^{j+1}}\frac{1 }{(1+2^{-2k}|x-y-r\theta|^2)^N}dr d\theta.
\end{eqnarray*}
This leads to our desired conclusion.
\end{proof}
\begin{lemma}\label{lem4.5} For $j,\,m\in\mathbb{Z}$, $\nu\in\Lambda_s$ and $s\in\mathbb{N}$ with $s\geq 100$, let $F_{j,\,\nu,\,m}^s(x,\,y)$ be the kernel of
the operator $V_mT_{j\nu}^s$. Let $Q$ be a cube with $\ell(Q)=2^{j-s}$. Then for any $x\in\mathbb{R}^d$, $y,\,y_0\in Q$,
\begin{eqnarray*}|F_{j,\nu,m}^s(x,y)-F_{j,\nu,m}^s(x,y_0)|&\lesssim & 2^{-s-m-md}\|\Omega\|_{L^{\infty}(S^{d-1})}\int_{S^{d-1}}\Gamma_{\nu}^s(\theta)\int^{2^{j+1}}_{2^{j-1}}\\
&&\Big[\int^1_0\frac{N2^{-m}|x-(ty+(1-t)y_0)-r\theta|dt}{(1+2^{-2m}|x-(ty+(1-t)y_0)-r\theta|^2)^{N+1}}\\
&&+\frac{1}{(1+2^{-2m}|x-y-r\theta|^2\big)^N}\Big]\phi_j(r)drd\theta.
\end{eqnarray*}
\end{lemma}
\begin{proof}The proof is essentially given in \cite[Section 4.3]{dinglai}. For the sake of self-contained, we include a concise proof. By integrating by parts, we   write
\begin{eqnarray*}
F_{j,\,\nu,m}^s(x,\,y)&=&\frac{1}{(2\pi)^d}\int_{\mathbb{R}^d}\int_{\mathbb{R}^d}{\rm e}^{i\langle x-y-z,\xi\rangle}\eta(2^m\xi)d\xi K_{j\nu}^s(z)dz\\
&=&\frac{1}{(2\pi)^d}\int_{S^{d-1}}\Omega(\theta)\Gamma_{\nu}^s(\theta)\Big\{
\int^{\infty}_0\int_{\mathbb{R}^d}{\rm e}^{i\langle x-y-r\theta,\xi\rangle}\\
&&\times\frac{(I-2^{-2m}\Delta_{\xi})^N\eta(2^m\xi)}{(1+2^{-2m}|x-y-r\theta|^2)^N}d\xi \phi_j(r)r^{-1}dr\Big\}d\theta.\end{eqnarray*}
Let
\begin{eqnarray*}{\rm D}_1(x,\,y,y_0)&=&\frac{1}{(2\pi)^d}\int_{S^{d-1}}\Omega(\theta)\Gamma_{\nu}^s(\theta)\Big\{
\int^{\infty}_0\int_{\mathbb{R}^d}({\rm e}^{i\langle-y,\xi\rangle}-{\rm e}^{i\langle-y_0,\xi\rangle}){\rm e}^{i\langle x-r\theta,\xi\rangle}\\
&&\times\frac{(I-2^{-2m}\Delta_{\xi})^N\eta(2^m\xi)}{(1+2^{-2m}|x-y-r\theta|^2)^N}d\xi \phi_j(r)r^{-1}dr\Big\}d\theta,
\end{eqnarray*}
and
\begin{eqnarray*}&&{\rm D}_2(x,y,y_0)=\frac{1}{(2\pi)^d}\int_{S^{d-1}}\Omega(\theta)\Gamma_{\nu}^s(\theta)\Big\{
\int^{\infty}_0\int_{\mathbb{R}^d} {\rm e}^{i\langle x-y_0-r\theta,\xi\rangle}\\
&&\qquad \times(I-2^{-2m}\Delta_{\xi})^N\eta(2^m\xi)\Big(\frac{1}{\Upsilon(x,y,m)}-\frac{1}{\Upsilon(x,y_0,m)}\Big)d\xi \phi_j(r)r^{-1}dr\Big\}d\theta,
\end{eqnarray*}
where $\Upsilon(x,y,\,m)=(1+2^{-2m}|x-y-r\theta|^2)^N$. We then have that
$$F_{j,\nu,m}^s(x,y)-F_{j,\nu,m}^s(x,y_0)={\rm D}_1(x,\,y,y_0)+{\rm D}_2(x,\,y,y_0).
$$
By the facts that
$$|{\rm e}^{i\langle-y,\xi\rangle}-{\rm e}^{i\langle-y_0,\xi\rangle}|\le |y-y_0||\xi|
$$
and\begin{eqnarray}\label{equa3.nabla}
|(I-2^{-2m}\Delta_{\xi})^N\eta(2^m\xi)|\lesssim \chi_{\{|\xi|\leq 2^{-m}\}}(\xi),
\end{eqnarray}
it follows that
\begin{eqnarray*}|{\rm D}_1(x,y,y_0)|\lesssim2^{j-s-m-md}\int_{S^{d-1}}|\Omega(\theta)\Gamma_{\nu}^s(\theta)|
\int^{\infty}_0\frac{\phi_j(r)r^{-1}dr}{(1+2^{-2m}|x-y-r\theta|^2)^N}d\theta.
\end{eqnarray*}
On the other hand, a trivial computation shows that
\begin{eqnarray*}
&&|\Upsilon(x,y,\,m)-\Upsilon(x,y_0,\,m)|=\Big|\int^1_0\langle\langle y-y_0, \nabla \Upsilon(ty+(1-t)y_0)\rangle\Big|\\
&&\quad\lesssim |y-y_0|2^{-m}\int^1_0\frac{N2^{-m}|x-(ty+(1-t)y_0)-r\theta|}{(1+2^{-2m}|x-(ty+(1-t)y_0)-r\theta|^2)^{N+1}}dt.
\end{eqnarray*}
This, along with inequality (\ref{equa3.nabla}), yields
\begin{eqnarray*}|{\rm D}_2(x,y,y_0)|&\lesssim & 2^{-s-m-md}\|\Omega\|_{L^{\infty}(S^{d-1})}\int_{S^{d-1}}\Gamma_{\nu}^s(\theta)\int^{2^{j+1}}_{2^{j-1}}\\
&& \int^1_0\frac{N2^{-m}|x-(ty+(1-t)y_0)-r\theta|}{(1+2^{-2m}|x-(ty+(1-t)y_0)-r\theta|^2)^{N+1}}dt \phi_j(r)drd\theta.
\end{eqnarray*}
Combining estimates for ${\rm D}_1(x,y,y_0)$ and ${\rm D}_2(x,y,y_0)$ then finishes the proof of Lemma \ref{lem4.5}.
\end{proof}

We are now ready to prove (\ref{equa3.u3}).

\medskip

{\it Proof of the inequality (\ref{equa3.u3}).} We choose $\gamma\in (0,\,\frac{1}{8d})$,  $N_0 \in\mathbb{N}$  such that
\begin{eqnarray}\label{equa3.const}N_0\gamma>16,\,\,\big[\frac{1}{3}-\gamma (3d/2+1)\big]N_0>\frac{3}{2}.\end{eqnarray}
For $b\in {\rm BMO}(\mathbb{R}^d)$ and $m=j-\lfloor \frac{s}{2}\rfloor$, write
\begin{eqnarray*}&&\big(b(x)-\langle b\rangle_Q\big)T^i_{j}h_Q(x)-\sum_{\nu}G_{\nu}^s\big[\big(b-\langle b\rangle_Q\big)T^{i,s}_{j\nu}h_Q\big](x)\\
&&\quad=\sum_{\nu}(I-G_{\nu}^s)\big[\big(b-\langle b\rangle_Q\big)V_mT^{i,s}_{j\nu}h_Q\big](x)\\
&&\qquad+\sum_{\nu}G_{\nu,b}^sT^{i,s}_{j\nu}h_Q(x)-\sum_{\nu}G_{\nu,b}^s (V_mT^{i,s}_{j\nu}h_Q )(x)\\
&&\qquad+\big(b(x)-\langle b\rangle_Q\big)\sum_{\nu}(I-G_{\nu}^s)\Big[\sum_{k<m}W_kT^{i,s}_{j\nu}h_Q\Big](x)\\
&&\quad=:{\rm U}_{31,j}^{i,s}h_Q(x)+{\rm U}_{32,j}^{i,s}h_Q(x)+{\rm U}_{33,j}^{i,s}h_Q(x)+{\rm U}_{34,j}^{i,s}h_Q(x).
\end{eqnarray*}
As in the estimate for ${\rm U}_2$, it follows from (\ref{equation3.7}) in Lemma \ref{lem4.3} that
\begin{eqnarray}\label{equation3.16}
&&\Big|\Big\{x\in\mathbb{R}^d\backslash E:\, \Big|\sum_{i=0}^{\infty}\sum_{s> N_0i}\sum_{j\in\mathbb{Z}}\sum_{Q\in\mathcal{S}_{j-s}}{\rm U}_{32,j}^{i,s}h_Q(x)\Big|>\frac{1}{48}\Big\}\Big|\\
&&\quad\leq \Big(\sum_{i=0}^{\infty}\sum_{s> N_0i}\Big\|\sum_{j\in\mathbb{Z}}\sum_{Q\in\mathcal{S}_{j-s}}\sum_{\nu}G_{\nu,\,b}^sT_{j\nu}^{i,\,s}h_Q\Big\|_{L^2(\mathbb{R}^d\backslash E)}\Big)^2\lesssim \int_{\mathbb{R}^d}|f(x)|dx,\nonumber
\end{eqnarray}
and from (\ref{equation3.8}) in Lemma \ref{lem4.3} that
\begin{eqnarray}\label{equation3.17}
&&\Big|\Big\{x\in\mathbb{R}^d\backslash E:\, \Big|\sum_{i=0}^{\infty}\sum_{s> N_0i}\sum_{j\in\mathbb{Z}}\sum_{Q\in\mathcal{S}_{j-s}}{\rm U}_{33,j}^{i,s}h_Q(x)\Big|>\frac{1}{48}\Big\}\Big|\\
&&\quad\leq \Big(\sum_{i=0}^{\infty}\sum_{s> N_0i}\Big\|\sum_{j\in\mathbb{Z}}\sum_{Q\in\mathcal{S}_{j-s}}\sum_{\nu}G_{\nu,\,b}^sV_mT_{j\nu}^{i,\,s}h_Q\Big\|_{L^2(\mathbb{R}^d\backslash E)}\Big)^2\lesssim \int_{\mathbb{R}^d}|f(x)|dx.\nonumber
\end{eqnarray}

We now estimate ${\rm U}_{34,j}^{i,s}h_Q$.
For each cube $Q\in\mathcal{S}_{j-s}$,  $y\in Q$, $2^{j-1}\leq r\leq 2^{j+1}$, and $ \theta\in S^{d-1}$, denote by $Q_{y+r\theta,\,k}$ the cube centered at $y+r\theta$ and having side length $2^k$.
We have that
\begin{eqnarray*}
\big|\langle b\rangle_Q-\langle b\rangle_{Q_{y+r\theta,\,k}}\big|&\lesssim&\big|\langle b\rangle_Q-\langle b\rangle_{Q_{y+r\theta,\,j-s}}\big|+
\big|\langle b\rangle_{Q_{y+r\theta,\,k}}-\langle b\rangle_{Q_{y+r\theta,\,j-s}}\big|\\
&\lesssim &|k-j|+s.
\end{eqnarray*}
Recall that $N=\lfloor d/2\rfloor+1$. A trivial computation yields
$$\int_{\mathbb{R}^d}\frac{|b(x)-\langle b\rangle_{Q(y+r\theta,\,k)}|}{(1+2^{-2k}|x-y-r\theta|^2)^N}dx\lesssim 2^{kd}.$$
This, in turn, implies that
\begin{eqnarray}\label{equa5.11}
\int_{\mathbb{R}^d}\frac{|b(x)-\langle b\rangle_Q|}{(1+2^{-2k}|x-y-r\theta|^2)^N}dx&\leq & \int_{\mathbb{R}^d}\frac{|b(x)-\langle b\rangle_{Q(y+r\theta,\,k)}|}{(1+2^{-2k}|x-y-r\theta|^2)^N}dx\\
&&+\int_{\mathbb{R}^d}\frac{|k-j|+s}{(1+2^{-2k}|x-y-r\theta|^2)^N}dx\nonumber\\
&\lesssim&2^{kd}(|k-j|+s).\nonumber
\end{eqnarray}
From Lemma \ref{lem4.4}, we deduce that
\begin{eqnarray*}
&&\|(b-\langle b\rangle_Q)(I-G_{\nu}^s)W_kT_{j\nu}^{i,s}h_Q\|_{L^1(\mathbb{R}^d)}\\
&&\quad\lesssim 2^{(s\gamma+k)N_1+2s\gamma N}2^{-j(N_1+1)}2^{-kd}\|\Omega_i\|_{L^{\infty}(S^{d-1})}\int_{\mathbb{R}^d}\int_{S^{d-1}}|\Gamma_{\nu}^s(\theta)|\nonumber\\
&&\qquad\times \int_{2^{j-1}}^{2^{j+1}}\int_{\mathbb{R}^d}\frac{|b(x)-\langle b\rangle_Q|}{(1+2^{-2k}|x-y-r\theta|^2)^N}dx dr d\theta|h_Q(y)|dy \nonumber\\
&&\quad\lesssim 2^{-s\gamma(d-1)}2^{(-j+k)N_1}2^{s\gamma(N_1+2N)}\|\Omega_i
\|_{L^{\infty}(S^{d-1})}(|k-j|+s)\|h_Q\|_{L^1(\mathbb{R}^d)},\nonumber
\end{eqnarray*}
since $\|\Gamma_{\nu}^s\|_{L^1(S^{d-1})}\lesssim 2^{-\gamma s(d-1)}$. Note that for $s\in \mathbb{N}$, $s\lesssim 2^{\gamma s}$. Then
\begin{eqnarray}\label{equation3.18}
&&\Big|\Big\{x\in\mathbb{R}^d\backslash E:\, \Big|\sum_{i=0}^{\infty}\sum_{s> Ni}\sum_{j\in\mathbb{Z}}\sum_{Q\in\mathcal{S}_{j-s}}{\rm U}_{34,j}^{i,s}h_{Q} (x)\Big|>\frac{1}{48}\Big\}\Big|\\
&&\quad\lesssim \sum_{i=0}^{\infty}\sum_{s> N_0i}\sum_{j\in\mathbb{Z}}\sum_{\nu}\sum_{k<m}\sum_{Q\in\mathcal{S}_{j-s}}\|(b-\langle b\rangle_Q)(I-G_{\nu}^s)W_kT_{j\nu}^{i,s}h_Q\|_{L^1(\mathbb{R}^d)}\nonumber\\
&&\quad\lesssim \sum_{i=0}^{\infty}2^i\sum_{s> N_0i}2^{s\gamma(N_1+2N)}\sum_{j}\sum_{k<m}2^{(-j+k)N_1}(j-k+s)\sum_{Q\in \mathcal{S}_{j-s}}\|h_Q\|_{L^1(\mathbb{R}^d)}\nonumber\\
&&\quad \lesssim\sum_{i=0}^{\infty}2^i\sum_{s> N_0i}2^{s\gamma(N_1+2N)}2^{s\gamma}2^{-N_1s/2}\sum_{Q}\|h_Q\|_{L^1(\mathbb{R}^d)}\lesssim \int_{\mathbb{R}^d}|f(x)|dx,\nonumber
\end{eqnarray}
provided we choose $N_1\in\mathbb{N}$ in Lemma \ref{lem4.4} such that
\begin{eqnarray}\label{equa4.const}
N_0\big(N_1/2-\gamma N_1-2N\gamma-\gamma)>1.
\end{eqnarray}

It remains to consider term $(I-G_{\nu}^s)(b-\langle b\rangle_{Q})V_mT_{j\nu}^{i,s}h_{Q}(x)$. For each cube $Q\in \mathcal{S}_{j-s}$, let $y_Q$ be the center of $Q$. Applying Lemma \ref{lem4.5} and the vanishing moment of $h_Q$, we get that for each $Q\in \mathcal{S}_{j-s}$,
\begin{eqnarray*}&&\big\|(b-\langle b\rangle_{Q})V_mT_{j\nu}^{i,s}h_{Q}\big\|_{L^1(\mathbb{R}^d)}\\
&&\quad\lesssim\int_{\mathbb{R}^d}\int_{\mathbb{R}^d}|F_{j,\,\nu,\,m}^s(x,\,y)-F_{j,\,\nu,\,m}^s(x,\,y_Q)||b(x)-\langle b\rangle_Q|dx|h_Q(y)|dy\\
&&\quad\lesssim 2^{-s-m-md}\|\Omega_i\|_{L^{\infty}(S^{d-1})}\int_{\mathbb{R}^d}\int_{S^{d-1}}\Gamma_{\nu}^s(\theta)\int^{\infty}_0\\
&&\qquad\Big\{\int_{\mathbb{R}^d}|b(x)-\langle b\rangle_Q|\int^1_0\frac{N2^{-m}|x-(ty+(1-t)y_Q)-r\theta|}{(1+2^{-2m}|x-(ty+(1-t)y_Q)-r\theta|^2)^{N+1}}dtdx\\
&&\qquad+\int_{\mathbb{R}^d}
\frac{|b(x)-\langle b\rangle_Q|}{(1+2^{-2m}|x-y-r\theta|^2\big)^N}dx\Big\}\phi_j(r)drd\theta|h_Q(y)|dy.
\end{eqnarray*}
As in inequality (\ref{equa5.11}), we have that for each $t\in [0,\,1]$,
$$
\int_{\mathbb{R}^d}|b(x)-\langle b\rangle_Q|\frac{N2^{-m}|x-(ty+(1-t)y_Q)-r\theta|}{(1+2^{-2m}|x-(ty+(1-t)y_Q)-r\theta|^2)^{N+1}}dx\\
\lesssim
2^{md}s.$$
This, along with (\ref{equa5.11}), gives us that
$$\big\|(b-\langle b\rangle_{Q})V_mT_{j\nu}^{i,s}h_{Q}\big\|_{L^1(\mathbb{R}^d)}\lesssim
2^{\lfloor \frac{s}{2}\rfloor-s}s2^{-s\gamma(d-1)}\|\Omega_i\|_{L^{\infty}(S^{d-1})}\|h_Q\|_{L^1(\mathbb{R}^d)}.
$$
Let $L^{i,s}_{\nu,m}(x)=\sum_{j\in\mathbb{Z}}\sum_{Q\in \mathcal{S}_{j-s}}(b-\langle b\rangle_{Q})V_mT_{j\nu}^{i,s}h_{Q}(x)$. We then have
\begin{eqnarray*}
&&\sum_{i=0}^{\infty}2^{i/2}\sum_{s> N_0i}2^{\frac{s}{6}}\sum_{\nu}2^{s\gamma(d-1)}2^{s\gamma(\frac{d}{2}+1)}\|L_{\nu,m}^{i,s}\|_{L^1(\mathbb{R}^d)}\\
&&\quad\lesssim \sum_{i=0}^{\infty}2^{\frac{3i}{2}}\sum_{s> N_0i}2^{\frac{s}{6}}2^{-\frac{s}{2}}2^{s\gamma d}2^{s\gamma(\frac{d}{2}+1)}\sum_{Q}\|h_Q\|_{L^1(\mathbb{R}^d)}\lesssim\|f\|_{L^1(\mathbb{R}^d)},
\end{eqnarray*}
since $N_0$ and $\gamma$ satisfies (\ref{equa3.const}).
It follows from the pigeonhole principle, inequality (\ref{eq3.multiplier}) and Lemma
\ref{lem4.1}  that for some constant $C_0$,
\begin{eqnarray}\label{equation3.19}
&&\Big|\Big\{x\in\mathbb{R}^d\backslash E:\,\Big|\sum_{i=0}^{\infty}\sum_{s> N_0i}\sum_{j\in\mathbb{Z}}\sum_{Q\in\mathcal{S}_{j-s}}{\rm U}_{31,j}^{i,s}h_Q(x)\Big|>\frac{1}{48}\Big\}\Big|\\
&&\quad\lesssim \sum_{i=0}^{\infty}\sum_{s>N_0i}\sum_{\nu}\Big|\Big\{x\in\mathbb{R}^d\backslash E:\,|(I-G_{\nu}^s)L_{\nu,\,m}^{i,s}(x)|>C_02^{-i/2-s/6-\gamma s(d-1)}\Big\}\Big|\nonumber\\
&&\quad\lesssim \sum_{i=0}^{\infty}2^{i/2}\sum_{s> N_0i}2^{\frac{s}{6}}\sum_{\nu}2^{s\gamma(d-1)}2^{s\gamma(\frac{d}{2}+1)}\|L_{\nu,m}^{i,s}\|_{L^1(\mathbb{R}^d)}\lesssim\|f\|_{L^1(\mathbb{R}^d)}.\nonumber
\end{eqnarray}
Combining estimates (\ref{equation3.16}), (\ref{equation3.17}), (\ref{equation3.18}) and (\ref{equation3.19}) yields
\begin{eqnarray*}
|\{x\in\mathbb{R}^d\backslash E:\,|{\rm U}_3h(x)|>\frac{1}{12}\}|\lesssim \|f\|_{L^1(\mathbb{R}^d)}.
\end{eqnarray*}
This completes the proof of (\ref{equa3.u3}).\qed

\medskip

{\bf Acknowledgement} The authors would like to express their sincerely
thanks to the referee for his/her valuable remarks and suggestions, which made this paper more readable.
Also,
The authors would like to thank Dr. Israel. P. Rivera-R\'{i}os for helpful comments.

\bibliographystyle{amsplain}

\begin{thebibliography}{99}

\bibitem{abkp} J. Alvarez, R. J. Babgy, D. Kurtz and C. P\'erez, Weighted estimates for commutators of linear operators, Studia Math.
\textbf{104} (1993), 195-209.


\bibitem{cz1}A. P. Calder\'on and A. Zygmund, On the existence of certain singular integrals, Acta Math. \textbf{88} (1952), 85-139.

\bibitem{cz2} A. P. Calder\'on and A. Zygmund, On  singular integrals, Amer. J. Math. \textbf{78} (1956), 289-309.



\bibitem{chr2} M. Christ and J. L. Rubio de Francia, Weak type (1,\,1) bounds for rough operators, II, Invent. Math. \textbf{93} (1988), 225-237.

\bibitem{crw} R. R. Coifman, R. Rochberg and G. Weiss, Factorization theorems for Hardy
spaces in several variables, Ann. of Math., \textbf{103} (1976), 611-635.

\bibitem{dinglai} Y. Ding and X. Lai, Weak type $(1,\,1)$ bounded criterion for singular integral with rough kernel and its applications,
Trans. Amer. Math. Soc.  \textbf{371} (2019),  1649-1675.

\bibitem{duo} J. Duoandikoetxea, Weighted norm inequalities for homogeneous singular integrals, Trans. Amer.
Math. Soc. \textbf{336} (1993), 869-880.

\bibitem{drf} J. Duoandikoetxea and J. L. Rubio de Francia, Maximal and singular integrals via Fourier transform estimates, Invent. Math. \textbf{84} (1986), 541-561.

\bibitem {gra}L. Grafakos, Modern Fourier Analysis, Graduate Texts in Mathematics, Vol. \textbf{250} (Third edition), Springer, New York, 2014.

\bibitem{hu1} G. Hu, $L^p(\mathbb{R}^n)$ boundedness for the commutator of a homogeneous singular integral, Studia Math. \textbf{154} (2003), 13-47.

\bibitem{hsw}G. Hu, Q. Sun and X. Wang,  $L^p(\mathbb{R}^n)$ bounds for commutators of
convolution operators, Colloquium Math. \textbf{93} (2002), 11-20.

\bibitem{hp} T.  Hyt\"onen and C. P\'erez, Sharp weighted bounds involving $A_1$, Anal. PDE. \textbf{6} (2013), 777-818.

\bibitem{kurtzwh} D. Kurtz and R. L. Wheeden, Results on weighted norm inequalities for multiplier, Trans. Amer. Math. Soc. \textbf{255} (1979), 343-362.

\bibitem{ler4} A. K. Lerner, A weak type estimates for rough singular integrals, Rev. Mat. Iberoam. \textbf{35} (2019),  1583-1602.

\bibitem{lth} J. Lan, X. Tao and G. Hu, Weak type endpoint estimates for the commutators of rough singular integral operators,
Math. Inequal. Appl. \textbf{23} (2020), 1179-1195.

\bibitem{mucken}B. Muckenhoupt and R. L. Wheeden, Weighted norm inequalities for singular and fractional
integrals, Trans. Amer. Math. Soc. 161 (1971), 249-258.

\bibitem{oneil} R. O'Neil, Fractional integration in Orlicz spaces, Trans. Amer. Math. Soc. 115 (1965),
300-328.

\bibitem{perez1} C. P\'erez,  Endpoint estimates for commutators of singular integral operators, J. Funct. Anal. \textbf{128}
(1995), 163-185.

\bibitem{rr}M.  Rao and Z. Ren, Theory of Orlicz spaces, Monographs and Textbooks in Pure
and Applied Mathematics, 146, Marcel Dekker Inc., New York, 1991.


\bibitem{rw}F. Ricci and G. Weiss,  A characterization of $H^1(S^{n-1})$, Proc. Sympos. Pure Math. of Amer. Math. Soc., (S. Wainger and G. Weiss eds), Vol 35 I(1979), 289-294.

\bibitem{se} A. Seeger,  Singular integral operators with rough convolution kernels, J. Amer. Math. Soc. \textbf{9} (1996), 95-105.

\bibitem{stwe} E. M. Stein and G. Weiss, Interpolation of operators with change of measures, Trans.
Amer. Math. Soc. \textbf{87} (1958), 159-172.

\end{thebibliography}

\end{document}